\documentclass[a4paper]{article}
\usepackage{geometry}
\usepackage{amssymb}
\usepackage{latexsym}
\usepackage{amsmath}
\usepackage{indentfirst}
\usepackage{graphicx}
 \usepackage[colorlinks=true]{hyperref}
\usepackage{mathrsfs} 
\usepackage{chngcntr}
\usepackage{color}

\newtheorem{Thm}{Theorem}[section]

\newtheorem{Lem}{Lemma}[section]
\newtheorem{Pro}{Proposition}[section]

\newtheorem{Rem}{Remark}[section]

\numberwithin{equation}{section}

\newenvironment{proof}{\medskip\par\noindent{\bf Proof\/}:\quad}{\qquad
\raisebox{-0.5mm}{\rule{1.5mm}{1mm}}\vspace{6pt}}
\begin{document}
\title{A singularly perturbed fractional Kirchhoff problem}

\author{
{Vicen\c tiu D. R\u{a}dulescu}\thanks{Email: {\tt radulescu@inf.ucv.ro}}\\
\small Faculty of Applied Mathematics, AGH University of Science and Technology, 30-059 Krak\'{o}w, Poland\\
\small Department of Mathematics, University of Craiova, Craiova, 200585, Romania\\
{Zhipeng Yang}\thanks{Corresponding author. Email: {\tt yangzhipeng326@163.com}}\\
\small Mathematical Institute, Georg-August-University of G\"ottingen, G\"ottingen 37073, Germany\\
}

\date{} \maketitle

\textbf{Abstract.}
In this paper, we first establish the uniqueness and non-degeneracy of positive solutions to the fractional Kirchhoff problem
\begin{equation*}
\Big(a+b{\int_{\mathbb{R}^{N}}}|(-\Delta)^{\frac{s}{2}}u|^2dx\Big)(-\Delta)^su+mu=|u|^{p-2}u,\quad \text{in}\ \mathbb{R}^{N},
\end{equation*}
where $a,b,m>0$, $0<\frac{N}{4}<s<1$, $2<p<2^*_s=\frac{2N}{N-2s}$ and $(-\Delta )^s$ is the fractional Laplacian. Then, combining this non-degeneracy result and Lyapunov-Schmidt reduction method, we derive the existence of semiclassical solutions to the singularly perturbation problem
\begin{equation*}
\Big(\varepsilon^{2s}a+\varepsilon^{4s-N} b{\int_{\mathbb{R}^{N}}}|(-\Delta)^{\frac{s}{2}}u|^2dx\Big)(-\Delta)^su+V(x)u=|u|^{p-2}u,\quad \text{in}\ \mathbb{R}^{N},
\end{equation*}
for $\varepsilon> 0$ sufficiently small and a potential function $V$.

\vspace{6mm} \noindent{\bf Keywords:} Fractional Kirchhoff equations; Uniqueness and nondegeneracy; Lyapunov-Schmidt reduction.

\vspace{6mm} \noindent


\section{Introduction and main results}

In this paper, we are concerned with the following singularly perturbed fractional Kirchhoff problem
\begin{equation}\label{eq1.1}
\Big(\varepsilon^{2s}a+\varepsilon^{4s-N} b{\int_{\mathbb{R}^{N}}}|(-\Delta)^{\frac{s}{2}}u|^2dx\Big)(-\Delta)^su+V(x)u=u^p,\quad \text{in}\ \mathbb{R}^{N},
\end{equation}
and the related unperturbed problem
\begin{equation}\label{eq1.2}
\Big(a+b{\int_{\mathbb{R}^{N}}}|(-\Delta)^{\frac{s}{2}}u|^2dx\Big)(-\Delta)^su+mu=u^p,\quad \text{in}\ \mathbb{R}^{N},
\end{equation}
where $a,b,m>0$, $\varepsilon>0$ is a parameter, $V: \mathbb{R}^{N} \rightarrow \mathbb{R}$ is a bounded continuous function,
$(-\Delta )^s$ is the fractional Laplacian and $p$ satisfies
\begin{equation*}
1<p<2_s^*-1=
\begin{cases}
\frac{N+2s}{N-2s}, &  0<s<\frac{N}{2}, \\
+\infty, &  s\geq \frac{N}{2},
\end{cases}
\end{equation*}
where $2^*_s$ is the standard fractional Sobolev critical exponent.
\par
If $s=1$, equation \eqref{eq1.2} reduces to the well known Kirchhoff type problem, which and their variants have been studied extensively in the literature. The equation that goes under the name of Kirchhoff equation was proposed in \cite{Kirchhoff1883} as a model for the transverse oscillation of a stretched string in the form
\begin{equation}\label{eq1.3}
\rho h \partial_{t t}^{2} u-\left(p_{0}+\frac{\mathcal{E}h}{2 L} \int_{0}^{L}\left|\partial_{x} u\right|^{2} d x\right) \partial_{x x}^{2} u=0,
\end{equation}
for $t \geq 0$ and $0<x<L$, where $u=u(t, x)$ is the lateral displacement at time $t$ and at position $x, \mathcal{E}$ is the Young modulus, $\rho$ is the mass density, $h$ is the cross section area, $L$ the length of the string, $p_{0}$ is the initial stress tension.
\par
Through the years, this model was generalized in several ways that can be collected in the form
\begin{equation*}
\partial_{t t}^{2} u-M\left(\|u\|^{2}\right) \Delta u=f(t, x, u), \quad x \in \Omega
\end{equation*}
for a suitable function $M:[0, \infty) \rightarrow \mathbb{R}$, called Kirchhoff function. The set $\Omega$ is a bounded domain of $\mathbb{R}^{N}$, and $\|u\|^{2}=\|\nabla u\|_{2}^{2}$ denotes the Dirichlet norm of $u$. The basic case corresponds to the choice
\begin{equation*}
M(t)=a+b t^{\gamma-1}, \quad a \geq 0, b \geq 0, \gamma \geq 1.
\end{equation*}
Problem \eqref{eq1.3} and its variants have been studied extensively in the literature. Bernstein obtains the global stability result in \cite{Bernstein1940BASUS}, which has been generalized to arbitrary dimension $N\geq 1$ by Poho\v{z}aev in \cite{Pohozaev1975MS}. We also point out that such problems may describe a process of some biological systems dependent on the average of itself, such as the density of population (see e.g. \cite{Arosio-Panizzi1996TAMS}). Many interesting work on Kirchhoff equations can be found in \cite{MR1161092,MR2270546,MR519648,MR1604955} and the references therein. From a mathematical point of view, the interest of studying Kirchhoff equations comes from the nonlocality of Kirchhoff type equations. For instance, if we take $M(t)=a+bt$, the consideration of the stationary analogue of Kirchhoff’s wave equation leads to elliptic problems
\begin{equation}\label{eq1.4}
\begin{cases}
-\Big(a+b\int_{\Omega}|\nabla u|^2dx\Big)\Delta u=f(x,u) &\text{in } \Omega, \\
u=0 &\text{on}\ \partial \Omega,
\end{cases}
\end{equation}
for some nonlinear functions $f(x,u)$. Note that the term $(\int_{\Omega}|\nabla u|^2dx)\Delta u$
depends not only on the pointwise value of $\Delta u$, but also on the integral of $|\nabla u|^2$ over the whole
domain. In this sense, Eqs. \eqref{eq1.1},  \eqref{eq1.2} and \eqref{eq1.4} are no longer the usual pointwise equalities. This
new feature brings new mathematical difficulties that make the study of Kirchhoff type equations
particularly interesting.
After Lions \cite{MR519648} introducing an abstract functional framework to this problem, this type of problem has received
much attention.
We refer to e.g. \cite{Chen-Kuo-Wu2011JDE,Perera-Zhang2006JDE,Shuai2015JDE} and to e.g. \cite{Deng-Peng-Shuai2015JFA,MR3218834,MR3360660,MR4293909,MR3412403,MR4305432,MR4021897,MR3200382} for mathematical researches on Kirchhoff type equations on bounded domains and in the whole space, respectively.
We also refer to \cite{MR3987384} for a recent survey of the results connected to this model.
\par
On the other hand, the interest in generalizing to the fractional case the model introduced by Kirchhoff does
not arise only for mathematical purposes. In fact, following the ideas of \cite{MR2675483} and the concept
of fractional perimeter, Fiscella and Valdinoci proposed in \cite{MR3120682} an equation describing the
behaviour of a string constrained at the extrema in which appears the fractional length of the rope. Recently, problem similar to \eqref{eq1.1} and \eqref{eq1.2} has been extensively investigated by many authors using different techniques and producing several relevant results  (see, e.g. \cite{MR4071927,MR4056169,MR4241293,MR4305431,MR4245633,MR3985380,Gu-Yang,MR4390819,R-Yang2,Yang2,MR4378099,Yang-Yu}).
\par
Besides, if $b=0$ in \eqref{eq1.2}, then we are led immediately to the following fractional Schr\"{o}dinger equation
\begin{equation}\label{eq1.5}
  a(-\Delta)^su+mu=u^p, \quad \text{in} \,\,\mathbb{R}^{N}.
\end{equation}
This equation is related to the standing wave solutions of the
time-independent fractional Schr\"{o}dinger equation
\begin{equation}\label{eq1.6}
 ih\frac{\partial \psi}{\partial t}=h^{2s}(-\Delta)^s\psi+V(x)\psi-f(x,|\psi|),\,\,\,\in\mathbb{R}^N\times\mathbb{R},
\end{equation}
where $h$ is the Plank constant and $V(x)$ is a potential function.
Eq. \eqref{eq1.6} was introduced by Laskin \cite{Laskin2000PLA,Laskin2002PR} as a fundamental equation of fractional quantum mechanics in the study of particles on stochastic fields modelled by L\'{e}vy process.
Since the fractional Laplacian $(-\Delta)^s$ is a nonlocal operator, one can not apply directly the
usual techniques dealing with the classical Laplacian operator.
Therefore, some ideas are proposed recently. In
\cite{Caffarelli-Silvestre2007PDE}, Caffarelli and Silvestre
expressed the operator $(-\Delta)^s$ on $\mathbb{R}^N$ as a generalized
elliptic BVP with local differential operators defined on the upper
half-space $\mathbb{R}^{N+1}_+=\{(t,x):t>0,x\in\mathbb{R}^N\}$. By means of
Lyapunov-Schmidt reduction, concentration phenomenon of solutions
was considered independently in \cite{Chen-Zheng2013CPAA} and
\cite{Davila-delPino-Wei2014JDE}. For more interesting results
concerning with the existence, multiplicity and concentration of
solutions for the fractional Laplacian equation, we refer reader to
\cite{Nezza-Palatucci-Valdinoci2012BSM,Valdinocibook2017} and the references therein.
\par
Uniqueness of ground states of nonlocal equations similar to Eq. \eqref{eq1.5} is of fundamental importance in the stability and blow-up analysis for solitary wave solutions of nonlinear dispersive equations, for example, of the generalized Benjamin-Ono equation. In contrast to the classical limiting case when $s=1$, in which standard ODE techniques are applicable, uniqueness of ground state solutions to Eq. \eqref{eq1.5} is a really difficult problem. In the case that $s=\frac{1}{2}$ and $N=1$, Amick and Toland \cite{MR1111746}, they obtained the uniqueness result for solitary waves of the Benjamin-Ono equation. After that, Lenzmann \cite{MR2561169} obtained the uniqueness of ground states for the pseudorelativistic Hartree equation in $3$-dimension. In \cite{MR3070568}, Frank and Lenzmann extends the results in \cite{MR1111746} to the case that $s \in(0,1)$ and $N=1$ with completely new methods. For the high dimensional case, Fall and Valdinoci \cite{MR3207007} established the uniqueness and nondegeneracy of ground state solutions of \eqref{eq1.5} when $s \in(0,1)$ is sufficiently close to $1$ and $p$ is subcritical. In their striking paper \cite{MR3530361}, Frank, Lenzmann and Silvestre solved the problem completely, and they showed that the ground state solutions of \eqref{eq1.5} with $a=m=1$ is unique for arbitrary space dimensions $N \geq 1$ and all admissible and subcritical exponents $p>0 .$ Moreover, they also established the nondegeneracy of ground state solutions in the sense that the linearized operator
\begin{equation*}
L_{+}=(-\Delta)^{s}+1+p|Q|^{p-1}
\end{equation*}
is nondegenerate; i.e., its kernel is given by
\begin{equation*}
\operatorname{ker} L_{+}=\operatorname{span}\left\{\partial_{x_{1}} Q, \cdots, \partial_{x_{N}} Q\right\}
\end{equation*}
where $Q \in H^{s}\left(\mathbb{R}^{N}\right)$ is a solution of \eqref{eq1.5} with $a=m=1$. For a systematical research on applications of nondegeneracy of ground states to perturbation problems, we refer to Ambrosetti and Malchiodi \cite{Ambrosetti-Malchiodibook} and the references therein.
\par
From the viewpoint of calculus of variation, the fractional Kirchhoff problem \eqref{eq1.2} is much more complex and difficult than the
classical fractional Laplacian equation \eqref{eq1.5} as the appearance of the term $b\big({\int_{\mathbb{R}^{N}}}|(-\Delta)^{\frac{s}{2}}u|^2dx\big)(-\Delta )^su$, which is of order four.
This fact leads to difficulty in obtaining the boundedness of the $(PS)$ sequence for the corresponding energy functional if $p\leq3$.
In this paper, we borrow some ideas from \cite{Wu-Gu,MR3862713,MR4161398} to overcome
 the restriction to $1<p\leq 3$ and will obtain the  uniqueness, nondegeneracy of ground state solution for
 problem \eqref{eq1.2}. Then, by Lyapunov-Schmidt reduction method, we also obtain some existence results for the singularly perturbation problem. One of the main idea is based on the scaling technique which allows us to find a relation between solutions of  \eqref{eq1.2} and the following equation
\begin{equation}\label{eq1.7}
(-\Delta)^{s} Q+Q=Q^{p} , \quad \text { in } \mathbb{R}^{N}
\end{equation}
where $0<s<1$ and $1<p<2_{s}^{*}-1$.
\par
In order to state our main results, we recall some preliminary
results for the fractional Laplacian. For $0<s<1$, the fractional
Sobolev space $H^s(\mathbb{R}^N)$ is defined by
\begin{equation*}
H^{s}(\mathbb{R}^N)=\bigg\{u\in L^2(\mathbb{R}^N):\frac{u(x)-u(y)}{|x-y|^{\frac{N}{2}+s}}\in L^2({\mathbb{R}^{N}\times\mathbb{R}^{N}})\bigg\},
\end{equation*}
endowed with the natural norm
\begin{equation*}
  \|u\|^2=\int_{\mathbb{R}^N}|u|^2dx+\int\int_{\mathbb{R}^N\times\mathbb{R}^N}\frac{|u(x)-u(y)|^2}{|x-y|^{N+2s}}dxdy.
\end{equation*}
\par
The fractional Laplacian $(-\Delta)^s$ is the pseudo-differential
operator defined by
\begin{equation*}
\mathcal{F}((-\Delta)^su)(\xi)=|\xi|^{2s}\mathcal{F}(u)(\xi),\ \ \xi\in \mathbb{R}^N,
\end{equation*}
where $\mathcal{F}$ denotes the Fourier transform. It is also given by
\begin{equation*}
(-\Delta)^{s}u(x)=-\frac{1}{2}C(N,s)\int_{\mathbb{R}^N}\frac{u(x+y)+u(x-y)-2u(x)}{|y|^{N+2s}}dy,
\end{equation*}
where
\begin{equation*}
C(N,s)=\bigg(\int_{\mathbb{R}^N}\frac{(1-cos\xi_1)}{|\xi|^{N+2s}}d\xi\bigg)^{-1},\ \xi=(\xi_1,\xi_2,...\xi_N).
\end{equation*}
From \cite{Nezza-Palatucci-Valdinoci2012BSM}, we have
\begin{equation*}
\|(-\Delta)^{\frac{s}{2}}u\|^2_2=\int_{\mathbb{R}^N}|\xi|^{2s}|\mathcal{F}(u)|^2d\xi=\frac{1}{2}C(N,s)
\int_{\mathbb{R}^N\times\mathbb{R}^N}\frac{|u(x)-u(y)|^2}{|x-y|^{N+2s}}dxdy
\end{equation*}
and the fractional Gagliardo-Nirenberg-Sobolve inequality
\begin{equation}\label{eq1.8}
\int_{\mathbb{R}^N}|u|^{p+1}dx\leq\mathcal{S}\Big(\int_{\mathbb{R}^N}|(-\Delta
)^{\frac{s}{2}}u|^2dx\Big)^{\frac{N(p-1)}{4s}}
\Big(\int_{\mathbb{R}^N}|u|^2dx\Big)^{\frac{p-1}{4s}(2s-N)+1},
\end{equation}
where $\mathcal{S}>0$ is the best constant. It follows from
\eqref{eq1.8} that
\begin{equation}\label{eq1.9}
J(u)=\frac{\mathcal{S}\Big(\int_{\mathbb{R}^N}|(-\Delta)^{\frac{s}{2}}u|^2dx\Big)^{\frac{N(p-1)}{4s}}\Big(\int_{\mathbb{R}^N}|u|^2dx\Big)^{\frac{p-1}{4s}(2s-N)+1}}
{\int_{\mathbb{R}^N}|u|^{p+1}dx}>0.
\end{equation}
It has been pointed out that in \cite{MR3530361} that $\mathcal{S}=\inf\limits_{u\in H^s(\mathbb{R}^N)\backslash\{0\}}J(u)$ is actually
attained. Here we say that $Q$ is a ground state solution if a solution of \eqref{eq1.7} is also a minimizer for $J(u)$.
\par
Despite all the results cited above, to the best of our knowledge, in literature there are still no articles summarizing the situation of different kind of solutions at different level of energy for the fractional Kirchhoff problem. We summarize our main results as follows.

\begin{Thm}\label{Thm1.1}Let $a,b,m>0$. Assume that $\frac{N}{4}<s<1$ and $1<p<2_s^*-1$. Then
equation \eqref{eq1.2} has a ground state solution $U\in H^s(\mathbb{R}^N)$,
\begin{itemize}
  \item[$(i)$] $U>0$ belongs to $C^\infty(\mathbb{R}^N)\cap H^{2s+1}(\mathbb{R}^N)$;
  \item[$(ii)$] there exist some $x_0\in \mathbb{R}^N$ such that
  $U(\cdot-x_0)$ is radial and strictly decreasing in $r=|x-x_0|$;
  \item[$(iii)$] there exist constants $C_1,C_2>0$ such that
\begin{equation*}
\frac{C_1}{1+|x|^{N+2s}}\leq U(x)\leq \frac{C_2}{1+|x|^{N+2s}},\quad \forall \ x\in \mathbb{R}^N.
\end{equation*}
  \end{itemize}
\end{Thm}

\begin{Thm}\label{Thm1.2}Under the assumptions of Theorem \ref{Thm1.1},  the ground state solution $U$ of \eqref{eq1.2} is unique up to
 translation. Moreover, $U$ is nondegenerate in $H^s(\mathbb{R}^N)$ in the sense that there holds
$$\ker L_+=span\{\partial_{x_1}U, \partial_{x_2}U,\cdots,  \partial_{x_N}U\},$$
where $L_+$ is defined as
\begin{equation*}
  L_+\varphi=\Big(a+b{\int_{\mathbb{R}^{N}}}|(-\Delta
)^{\frac{s}{2}}U|^2dx\Big)(-\Delta
)^s\varphi+m\varphi-pU^{p-1}\varphi+2b\Big({\int_{\mathbb{R}^{N}}}(-\Delta
)^{\frac{s}{2}}U(-\Delta )^{\frac{s}{2}}\varphi dx\Big)(-\Delta )^sU
\end{equation*}
acting on $L^2(\mathbb{R}^N)$ with domain $H^s(\mathbb{R}^N)$.
\end{Thm}
\par
By Theorem \ref{Thm1.2}, it is now possible that we apply Lyapunov-Schmidt reduction to study the perturbed fractional Kirchhoff equation \eqref{eq1.1}. We want to look for solutions of \eqref{eq1.1} in the Sobolev space $H^s(\mathbb{R}^N)$ for sufficiently small $\varepsilon$, which named semiclassical solutions. We also call such derived solutions as concentrating solutions since they will concentrate at certain point of the potential function $V$.
Moreover, it is expected that this approach can deal with problem \eqref{eq1.1} for all $1<p<2^*_s-1$, in a unified way.
To state our following results, let introduce some notations that will be used throughout the paper. For $\varepsilon>0$ and $y=\left(y_{1}, y_{2},\cdots y_{N}\right) \in \mathbb{R}^{N}$, write
\begin{equation*}
U_{\varepsilon, y}(x)=U\left(\frac{x-y}{ \varepsilon}\right), \quad x \in \mathbb{R}^{N}.
\end{equation*}
Assume that $V: \mathbb{R}^{N} \rightarrow \mathbb{R}$ satisfies the following conditions:
\begin{itemize}
  \item [$(V_1)$] $V$ is a bounded continuous function with $\inf\limits_{x \in \mathbb{R}^{N}} V>0$;
  \item [$(V_2)$] There exist $x_{0} \in \mathbb{R}^{N}$ and $r_{0}>0$ such that
\begin{equation*}
V\left(x_{0}\right)<V(x) \quad \text { for } 0<\left|x-x_{0}\right|<r_{0}
\end{equation*}
and $V \in C^{\alpha}\left(\bar{B}_{r_{0}}\left(x_{0}\right)\right)$ for some $0<\alpha<\frac{N+4 s}{2}$. That is, $V$ is of $\alpha$-th order H\"{o}lder continuity around $x_{0}$.
\end{itemize}
The assumption $(V_1)$ allows us to introduce the inner products
\begin{equation*}
\langle u, v\rangle_{\varepsilon}=\int_{\mathbb{R}^{N}}\left(\varepsilon^{2s} a (-\Delta)^{\frac{s}{2}} u \cdot(-\Delta)^{\frac{s}{2}} v+V(x) u v\right)dx
\end{equation*}
for $u, v \in H^{s}\left(\mathbb{R}^{N}\right)$. We also write
\begin{equation*}
H_{\varepsilon}=\left\{u \in H^{s}\left(\mathbb{R}^{N}\right):\|u\|_{\varepsilon}=\langle u, u\rangle_{\varepsilon}^{\frac{1}{2}}<\infty\right\}.
\end{equation*}
Now we state the existence result as follows.
\begin{Thm}\label{Thm1.3}
Under the assumptions of Theorem \ref{Thm1.1} and assume that $V$ satisfies $(V_1)$ and $(V_2)$. Then there exists $\varepsilon_{0}>0$ such that for all $\varepsilon \in\left(0, \varepsilon_{0}\right)$, problem \eqref{eq1.1} has a solution $u_{\varepsilon}$ of the form
\begin{equation*}
u_{\varepsilon}=U\left(\frac{x-y_{\varepsilon}}{\varepsilon}\right)+\varphi_{\varepsilon}
\end{equation*}
with $\varphi_{\varepsilon} \in H_{\varepsilon}$, satisfying
\begin{equation*}
\begin{aligned}
y_{\varepsilon} & \rightarrow x_{0}, \\
\left\|\varphi_{\varepsilon}\right\|_{\varepsilon} &=o\left(\varepsilon^{\frac{N}{2}}\right) \\
\end{aligned}
\end{equation*}
as $\varepsilon \rightarrow 0$ .
\end{Thm}

\par
This paper is organized as follows. We complete the proof of Theorem \ref{Thm1.1} in Section 2 and prove Theorem \ref{Thm1.2} in Section 3. In Section 4, we present some basic results which will be used later and explain the strategy of the proof of Theorem \ref{Thm1.3}. Finally, we finish the proof of Theorem \ref{Thm1.3} in Section 5.
\par
\vspace{3mm}
{\bf Notation.~}Throughout this paper, we make use of the following notations.
\begin{itemize}
\item[$\bullet$]  For any $R>0$ and for any $x\in \mathbb{R}^N$, $B_R(x)$ denotes the ball of radius $R$ centered at $x$;
\item[$\bullet$]  $\|\cdot\|_q$ denotes the usual norm of the space $L^q(\mathbb{R}^N),1\leq q\leq\infty$;
\item[$\bullet$]  $o_n(1)$ denotes $o_n(1)\rightarrow 0$ as $n\rightarrow\infty$;
\item[$\bullet$]  $C$ or $C_i(i=1,2,\cdots)$ are some positive constants may change from line to line.
\end{itemize}

\section{Proof of Theorem \ref{Thm1.1}}
In this section we prove Theorem \ref{Thm1.1}. Our methods depend on the following result due to Frank, Lenzmann and Silvestre \cite{MR3530361}.

\begin{Pro}\label{Pro2.1} Let $N\geq 1$, $0<s<1$ and $1<p<2_s^*-1$. Then
the following holds.
\begin{itemize}
  \item[$(i)$] there exists a minimizer $Q\in H^{s}(\mathbb{R}^N)$ for $J(u)$,  which can be chose a nonnegative function that solves equation  \eqref{eq1.7};
  \item[$(ii)$] there exist some $x_0\in \mathbb{R}^N$ such that   $Q(\cdot-x_0)$ is radial, positive and strictly decreasing in $r=|x-x_0|$. Moreover, the function $Q$ belongs to $C^\infty(\mathbb{R}^N)\cap H^{2s+1}(\mathbb{R}^N)$ and it satisfies
\begin{equation*}
\frac{C_1}{1+|x|^{N+2s}}\leq Q(x)\leq \frac{C_2}{1+|x|^{N+2s}},\quad\forall \ x\in \mathbb{R}^N,
\end{equation*}
with some constants $C_2\geq C_1>0$;
  \item[$(iii)$] $Q$ is a unique solution of \eqref{eq1.7} up to translation.
  \end{itemize}
\end{Pro}
{\bf Proof of Theorem \ref{Thm1.1}} Let $Q$ be the uniquely positive solution of \eqref{eq1.7} and also a minimizer of $J(u)$. Consider the
equation
\begin{equation}\label{eq2.1}
f(\mathcal{E})=\mathcal{E}-a-bm^{\frac{2}{p-1}+\frac{2s-N}{2s}}\|(-\Delta)^{\frac{s}{2}}Q\|^2_2\mathcal{E}^{\frac{N-2s}{2s}}=0,
\quad \mathcal{E}\in (a,+\infty).
\end{equation}
Recall that $s>\frac{N}{4}$. we have $\frac{N-2s}{2s}<1$, which implies that $\lim\limits_{\mathcal{E}\to +\infty}f(\mathcal{E})=+\infty$.
Moreover, one has $f(a)<0$. Consequently, there exists $\mathcal{E}_0>a$ such that $f(\mathcal{E}_0)=0$. Let
\begin{equation}\label{eq2.2}
U(x)=m^{\frac{1}{p-1}}Q(m^{\frac{1}{2s}}\mathcal{E}_0^{-\frac{1}{2s}}x)=m^{\frac{1}{p-1}}Q(\tilde{x}).
\end{equation}
It follows from the definition of $(-\Delta)^s$ that
\begin{equation*}
\Big((-\Delta)^sU\Big)(x)=m^{\frac{p}{p-1}}\mathcal{E}_0^{-1}\Big((-\Delta)^sQ\Big)(\tilde{x}),
\end{equation*}
which implies that $U$ is a positive solution of the following equation
\begin{equation}\label{eq2.3}
\mathcal{E}_0(-\Delta )^sU+mU=U^p \quad \text{in}\quad \mathbb{R}^{N}.
\end{equation}
Note that
\begin{equation*}
\|(-\Delta)^{\frac{s}{2}}U\|^2_2=m^{\frac{2}{p-1}+\frac{2s-N}{2s}}\|(-\Delta)^{\frac{s}{2}}Q\|^2_2\mathcal{E}_0^{\frac{N-2s}{2s}}.
\end{equation*}
We have $\mathcal{E}_0=a+b\|(-\Delta)^{\frac{s}{2}}U\|^2_2$. From \eqref{eq2.3}, we conclude that $U$ is a positive solution of
\eqref{eq1.2}. Furthermore, $U$ is a minimizer for $J(u)$ since $J(U)=J(Q)$. Then the conclusions of Theorem \ref{Thm1.1} follows
from \eqref{eq2.2} and Proposition \ref{Pro2.1}.

\begin{Rem}\label{R2.1}
Our method presented in the proof of Theorem \ref{Thm1.1} is not variational which is simpler and more useful. It's clear that the functional corresponding to equation \eqref{eq1.2} includes the 4-order term $\frac{b}{4}\|(-\Delta)^{\frac{s}{2}}u\|^4_2$. Thus it is not easy to show that the boundedness of a $(PS)$ sequence if $1<p\leq 3$.
However, our Theorem \ref{Thm1.1} implies that the 4-superlinear condition ($p>3$), which is always used in literature, seems not
essential for the existence of solutions for Kirchhoff type problems \eqref{eq1.1} and \eqref{eq1.2}.
\end{Rem}

\section{Proof or Theorem \ref{Thm1.2}}
\subsection{Uniqueness}
In this subsection we prove the uniqueness part of Theorem \ref{Thm1.2}.\\
{\bf Proof of Uniqueness:} Let $U$ be a ground state positive solution of \eqref{eq1.2} and set
\begin{equation*}
\mathcal{E}_0=a+b\|(-\Delta)^{\frac{s}{2}}U\|^2_2\quad \text{and}\quad \tilde{U}(x)=m^{-\frac{1}{p-1}}U(m^{-\frac{1}{2s}}\mathcal{E}_0^{\frac{1}{2s}}x).
\end{equation*}
As in the proof of Theorem \ref{Thm1.1}, it is easy to check that $\tilde{U}$ is a positive solution of \eqref{eq1.7} and a minimizer of
$J(u)$. Hence, $\tilde{U}$ is a ground state solution. It follows from Proposition \ref{Pro2.1} that there exists some $x_0\in \mathbb{R}^N$ such that $\tilde{U}(x)=Q(x-x_0)$, where $Q$ is the unique ground state of \eqref{eq1.7}. Consequently, we have
\begin{equation*}
{U}(x)=m^{\frac{1}{p-1}}Q\Big(m^{\frac{1}{2s}}\mathcal{E}_0^{-\frac{1}{2s}}x-x_0\Big).
\end{equation*}
and
\begin{equation}\label{eq3.1}
\mathcal{E}_0=a+bm^{\frac{2}{p-1}+\frac{2s-N}{2s}}\|(-\Delta)^{\frac{s}{2}}Q\|^2_2\mathcal{E}_0^{\frac{N-2s}{2s}}.
\end{equation}
\par
 In order to arrive at the desired conclusion, we have to prove  $\mathcal{E}_0$ is uniquely defined. For this, consider the real function $f$
 defined by \eqref{eq2.1}. Recall that $\frac{N-2s}{2s}<1$. It is evident that $f$ is strictly  increasing on $(a,+\infty)$ if $\frac{N-2s}{2s}\leq 0$. Moreover,  for $0<\frac{N-2s}{2s}<1$, one can verify that the solution  $\mathcal{E}_0$ of \eqref{eq2.1} must belong to the increasing
 interval  of $f$ since a solution $\mathcal{E}$ satisfies
\begin{equation*}
\mathcal{E}-bm^{\frac{2}{p-1}+\frac{2s-N}{2s}}\|(-\Delta)^{\frac{s}{2}}Q\|^2_2\mathcal{E}^{\frac{N-2s}{2s}}>0.
\end{equation*}
Hence, $\mathcal{E}_0$ is uniquely defined by \eqref{eq3.1}. The proof is completed.

\subsection{Nondegeneracy}
In this subsection we prove the nondegeneracy part of Theorem  \ref{Thm1.2}.  We start with introducing a nondegeneracy result due to Frank, Lenzmann and Silvestre \cite{MR3530361}.

\begin{Pro}\label{Pro3.1} Let $N\geq 1$, $0<s<1$, $1<p<2_s^*-1$ and $c$ be a positive constant.
Suppose that $Q\in H^s(\mathbb{R}^N)$ is a ground state solution of
\begin{equation}\label{eq3.2}
c(-\Delta)^sQ+Q=|Q|^p\quad \text{in}\quad \mathbb{R}^{N}
\end{equation}
and $T_+$ denotes the corresponding linearized operator given by
\begin{equation*}
T_+=c(-\Delta)^s+1-p|Q|^{p-1}.
\end{equation*}
Then the following holds.
\begin{itemize}
  \item[$(i)$] $Q$ is nondegenerate, i.e., $\ker T_+=span\{\partial_{x_1}Q, \partial_{x_2}Q,\cdots,
\partial_{x_N}Q\}$;
  \item[$(ii)$] the restriction of $T_+$ on $L^2_{rad}(\mathbb{R}^N)$ is one-to-one and thus it has an inverse $T_+^{-1}$ acting on
  $L^2_{rad}(\mathbb{R}^N)$;
  \item[$(iii)$] $T_+Q=-(p-1)Q^{p}$ and $T_+R=-2sQ$, where $R=\frac{2s}{p-1}Q+x\cdot \nabla Q$.
  \end{itemize}
\end{Pro}
\par
Let $U$ be a ground state solution of \eqref{eq1.2}. From Theorems \ref{Thm1.1} and the first part of Theorem \ref{Thm1.2}, $U\in H^{2s+1}(\mathbb{R}^N)$ and $U>0$ is radial. By rescaling $U(x)\mapsto \mu Q(\lambda x)$ with suitable
$\mu$ and $\lambda$, we can assume without loss of generality that $U$ satisfies the normalized equation
\begin{equation}\label{eq3.3}
\Big(a+b{\int_{\mathbb{R}^{N}}}|(-\Delta)^{\frac{s}{2}}U|^2dx\Big)(-\Delta)^sU+U=U^{p}\quad,\text{in}\quad\mathbb{R}^{N}.
\end{equation}
Then the linearized operator $L_+$ can be rewrote as
\begin{equation*}
L_+\varphi=\Big(a+b{\int_{\mathbb{R}^{N}}}|(-\Delta)^{\frac{s}{2}}U|^2dx\Big)(-\Delta)^s\varphi+\varphi-pU^{p-1}\varphi+2b\Big({\int_{\mathbb{R}^{N}}}(-\Delta
)^{\frac{s}{2}}U(-\Delta )^{\frac{s}{2}}\varphi dx\Big)(-\Delta )^sU
\end{equation*}
acting on $L^2(\mathbb{R}^N)$ with domain $H^s(\mathbb{R}^N)$.
\par
To complete the proof of Theorem \ref{Thm1.2}, we first show that the restriction of $L_+$ on $L^2_{rad}(\mathbb{R}^N)$ has trivial kernel.

\begin{Lem}\label{Lem3.1}
We have $(\ker L_+)\cap L^2_{rad}(\mathbb{R}^N)=\{0\}$.
\end{Lem}
\begin{proof}Assume that $v\in H^s(\mathbb{R}^N)\cap L^2_{rad}(\mathbb{R}^N)$ belongs to $\ker L_+$. Then we have
\begin{equation}\label{eq3.4}
\Big(a+b{\int_{\mathbb{R}^{N}}}|(-\Delta)^{\frac{s}{2}}U|^2dx\Big)(-\Delta)^sv+v-pU^{p-1}v=-2b\Big({\int_{\mathbb{R}^{N}}}(-\Delta
)^{\frac{s}{2}}U(-\Delta )^{\frac{s}{2}}v dx\Big)(-\Delta)^sU.
\end{equation}
Let $c=a+b\|(-\Delta)^{\frac{s}{2}}U\|^2_2$. Recall that $U$ is a ground state solution of \eqref{eq3.3}. It follows from Theorem
\ref{Thm1.2} that $c$ is a constant independent of $U$ under the assumptions of Theorem \ref{Thm1.1}. Hence, $U$ solves \eqref{eq3.2}
with $c=a+b\|(-\Delta)^{\frac{s}{2}}U\|^2_2$. We then can rewrite \eqref{eq3.4} as
\begin{equation}\label{eq3.5}
T_+v=-2b\Big({\int_{\mathbb{R}^{N}}}(-\Delta)^{\frac{s}{2}}U(-\Delta )^{\frac{s}{2}}v dx\Big)(-\Delta)^sU=-\frac{2b\sigma_v}{c}(-U+U^{p}),
\end{equation}
where
\begin{equation*}
\sigma_v={\int_{\mathbb{R}^{N}}}(-\Delta )^{\frac{s}{2}}U(-\Delta
)^{\frac{s}{2}}v dx.
\end{equation*}
By applying Proposition \ref{Pro3.1}, we conclude that
\begin{equation}\label{eq3.6}
v=-\frac{2b\sigma_v}{c}T_+^{-1}(-U+U^{p})=-\frac{b\sigma_v}{sc}\psi,
\end{equation}
where $\psi=x\cdot \nabla U$. Multiplying \eqref{eq3.6} by $(-\Delta)^sU$ and integrating over $\mathbb{R}^N$, we see that
\begin{equation}\label{eq3.7}
\int_{\mathbb{R}^{N}}v(-\Delta)^{s}Udx=-\frac{b\sigma_v}{sc}\int_{\mathbb{R}^{N}}\psi(-\Delta)^{s}Udx.
\end{equation}
Note that
\begin{equation}\label{eq3.8}
\int_{\mathbb{R}^{N}}v(-\Delta)^{s}Udx=\int_{\mathbb{R}^{N}}(-\Delta )^{\frac{s}{2}}U(-\Delta)^{\frac{s}{2}}v dx
\end{equation}
and
\begin{equation}\label{eq3.9}
\int_{\mathbb{R}^{N}}\psi(-\Delta)^{s}Udx=\frac{2s-N}{2}\int_{\mathbb{R}^{N}}|(-\Delta)^{\frac{s}{2}}U|^2 dx
\end{equation}
(see e.g. \cite{Ros2014The}). We then conclude from \eqref{eq3.7}-\eqref{eq3.9} that
\begin{equation*}
\sigma_v=-\frac{b(2s-N)\sigma_v}{2sc}\int_{\mathbb{R}^{N}}|(-\Delta
)^{\frac{s}{2}}U|^2 dx=-\frac{(c-a)(2s-N)}{2sc}\sigma_v,
\end{equation*}
which implies that $\sigma_v=0$ due to $-\frac{(c-a)(2s-N)}{2sc}\neq1$. From \eqref{eq3.5}, we have $v=0$. The proof is completed.
\end{proof}
\par
We next will adopt lines presented in \cite{MR3530361} and \cite{MR2561169} to obtain the nondegeneracy of $U$. In what follows, we assume without loss of generality that $N\geq 2$ holds.
\par
Since $(-\Delta )^sU=c^{-1}(-U+U^{p})$ and $U(x)=U(|x|)$ is radial function, the operator $L_+$ commutes with rotations in $\mathbb{R}^N$ (see e.g. \cite{MR2561169}). Therefore, we can decompose $L^2(\mathbb{R}^N)$ using spherical harmonics
\begin{equation*}
L^2(\mathbb{R}^N)=\oplus_{l\geq 0}\mathcal{H}_l
\end{equation*}
so that $L_+$ acts invariantly on each subspace
\begin{equation*}
\mathcal{H}_l=L^2(\mathbb{R}_+, r^{N-1}dr)\otimes\mathcal{Y}_l.
\end{equation*}
Here $\mathcal{Y}_l=span\{Y_{l,m}\}_{m\in M_l}$ denotes space of the spherical harmonics of degree $l$ in space dimension $N$ and $M_l$ is an index set depending on $l$ and $N$. Precisely, $M_l=\frac{(l+N-1)!}{(N-1)!l!}$ for $l\geq 0$ and $M_l=0$ for $l<0$.
Moreover, denote by $\Delta_{\mathbb{S}^{N-1}}$ the Laplacian-Beltrami operator on the unit $N-1$ dimensional sphere $\mathbb{S}^{N-1}$ in $\mathbb{R}^{N}$ and by $Y_{l,m},l=0,1,\cdots$ the spherical harmonics satisfy
\begin{equation*}
-\Delta_{\mathbb{S}^{N-1}}Y_{l,m}=\lambda_lY_{l,m}
\end{equation*}
for all $l=0,1,\cdots$ and $1\leq m\leq M_l-M_{l-2}$, where
\begin{equation*}
\lambda_l=l(N+l-2)\ \ \forall\ l\geq0
\end{equation*}
is an eigenvalue of $-\Delta_{\mathbb{S}^{N-1}}$ with multiplicity $\leq M_l-M_{l-2}$ for all $k\in\mathbb N$. In particular, $\lambda_0=0$ is of multiplicity $1$ with $Y_{0,1}=1$, and $\lambda_1=N-1$ is of multiplicity $N$ with $Y_{1,m}=x_m/|x|$ for $1\leq m\leq N$. (see e.g. Ambrosetti and Malchiodi \cite[Chapter 2 and Chapter 4]{Ambrosetti-Malchiodibook})
\par
We can describe the action of $L_+$ more precisely. For each $l$, the action of $L_+$ on the radial factor in $\mathcal{H}_l$ is given
by
\begin{equation*}
(L_{+,l}f)(r)=c\Big((-\Delta_l)^sf\Big)(r)+f(r)-pU^{p-1}(r)f(r)+2b(W_lf)(r)
\end{equation*}
with the nonlocal linear operator
\begin{equation*}
(W_lf)(r)=\frac{2\pi^{\frac{N}{2}}}{\Gamma(\frac{N}{2})}\Big((-\Delta_l)^sU\Big)(r)\int_0^{+\infty}\Big((-\Delta_l)^sU\Big)(r)r^{N-1}f(r)dr
\end{equation*}
for $f\in C_0^{\infty}(\mathbb{R}^+)\subset L^2(\mathbb{R}_+, r^{N-1}dr)$. Here
$(-\Delta_l)^s$ is given by spectral calculus and the known formula
\begin{equation*}
-\Delta_l=-\partial_r^2-\frac{N-1}{r}\partial_r+\frac{l(l+N-2)}{r^2}.
\end{equation*}
\par
Applying  arguments similar to that used in \cite{MR3530361} and \cite{MR2561169}, one can verify that each $L_{+,l}$ enjoys a Perron-Frobenius property, that is, if $E=\inf \sigma(\mathcal{H}_l)$ ia an eigenvalue, then $E$ is simple and the corresponding eigenfunction can be chosen strictly positive. Moreover, we have $L_{+,l}>0$ for $l\geq 2$ in the sense of quadratic forms (see e.g. \cite{MR2561169}).
\par
We now return to the proof of Theorem \ref{Thm1.2}.
\par
{\bf Proof of nondegeneracy: }By differentiating equation \eqref{eq3.2}, we see that $L_{+}\partial _{x_i}U=0$ for $i=1,2,\cdots, N$. Since $\partial _{x_i}U(x)=U'(r)\frac{x_i}{r}\in \mathcal{H}_1$, this shows that $L_{+,1}U'=0$. Note that $U'(r)<0$ by Theorem \ref{Thm1.1}. It follows from the Perron-Frobenius property  that $0$ is the lowest eigenvalue  of $L_{+,1}$, with $-U'(r)$ being its an corresponding eigenfunction. Therefore, for any $v\in \mathcal{H}_1$ satisfying $L_{+,1}v=0$ must be some linear combination of $\{\partial _{x_i}U:i=1,2,\cdots, N\}$. Recall that $L_{+,l}>0$ for $l\geq 2$. Applying the Perron-Frobenius property again, $0$ cannot be an eigenvalue of $L_{+,l}>0$ for $l\geq 2$. Finally, Lemma \ref{Lem3.1} implies that $L_{+,0}=\{0\}$. Consequently, for any $v\in \ker L_+$, we conclude that $v\in \mathcal{H}_1$ and hence $\ker L_+=\ker L_{+,1}=span\{\partial_{x_1}U, \partial_{x_2}U,\cdots,
\partial_{x_N}U\}$. The proof is completed.

\section{Some preliminaries}
In this section, we introduce Lyapunov-Schmidt reduction method of the proof of Theorem \ref{Thm1.3} and present some elementary estimates for later use.
\par
It is known that every solution to Eq. \eqref{eq1.1} is a critical point of the energy functional $I_{\varepsilon}: H_{\varepsilon} \rightarrow \mathbb{R}$, given by
\begin{equation*}
I_{\varepsilon}(u)=\frac{1}{2}\|u\|_{\varepsilon}^{2}+\frac{b\varepsilon^{4s-N}}{4}\left(\int_{\mathbb{R}^{N}}|(-\Delta)^{\frac{s}{2}} u|^{2}dx\right)^{2}-\frac{1}{p+1} \int_{\mathbb{R}^{N}} u^{p+1}dx
\end{equation*}
for $u \in H_{\varepsilon}$. It is standard to verify that $I_{\varepsilon} \in C^{2}\left(H_{\varepsilon}\right) .$ So we are left to find a critical point of $I_{\varepsilon}$.  However, due to the presence of the double nonlocal terms $(-\Delta)^s$ and $\left(\int_{\mathbb{R}^{N}}|(-\Delta)^{\frac{s}{2}} u|^{2}\right)$, it requires more careful estimates on the orders of $\varepsilon$ in the procedure. In particular, the nonlocal terms brings new difficulties in the higher order remainder term, which is more complicated than the case of the fractional Schr\"{o}dinger equation \eqref{eq1.5} or usual Kirchhoff equation \eqref{eq1.4}.
\par
Denote $Q$ be the solution of the following equation:
\begin{equation*}
\begin{cases}
(-\Delta)^{s} u+u=u^{p}, & \\
u(0)=\max\limits_{x \in \mathbb{R}^{N}} u(x), & \\
u>0, x \in \mathbb{R}^{N}.
\end{cases}
\end{equation*}
Then, it is easy to see that the function
\begin{equation*}
W_{\lambda}(x):=\lambda^{\frac{1}{p}} Q\left(\lambda^{\frac{1}{2 s}} x\right)
\end{equation*}
satisfies the equation
\begin{equation*}
(-\Delta)^{s} u+\lambda u=u^{p}, \quad x \in \mathbb{R}^{N}.
\end{equation*}
Therefore, for any point $\xi \in \mathbb{R}^{N}$, taking $\lambda=V(\xi)$, it follows that $V(\xi)^{\frac{1}{p-1}} W\left(V(\xi)^{\frac{1}{2 s}} x\right)$ satisfies
\begin{equation*}
\varepsilon^{2 s}(-\Delta)^{s} u+V(\xi) u=u^{p}, x \in \mathbb{R}^{N}.
\end{equation*}
By the same idea and the proof of Theorem \ref{Thm1.1}, we have known  $U(x)=m^{\frac{1}{p-1}}Q(m^{\frac{1}{2s}}\mathcal{E}_0^{-\frac{1}{2s}}x)$ is a positive unique solution of \eqref{eq1.2}. Moreover,
we have the polynomial decay instead of the usual exponential decay of $U$ of and its derivatives. That is,
\begin{equation*}
U(x)+|\nabla U(x)| \leq \frac{C}{1+|x|^{N+2s}}, \quad x \in \mathbb{R}^{N}
\end{equation*}
for some $C>0$.
\par
Following the idea from Cao and Peng \cite{MR2581983} (see also \cite{MR4021897}) , we will use the unique ground state $U$ of \eqref{eq1.2} with $m=V(x_0)$ to build the solutions of Eq. \eqref{eq1.1}. Since the $\varepsilon$-scaling makes it concentrate around $\xi$, this function constitutes a good positive approximate solution to \eqref{eq1.1}.
\par
For $\delta,\eta>0$, fixing $y \in B_{\delta}(x_0)$, we define
\begin{equation*}
M_{\varepsilon,\eta}=\left\{(y, \varphi): y \in B_{\delta}(x_0), \varphi\in E_{\varepsilon,y}, \|\varphi\|^2\leq\eta\varepsilon^{N}\right\}
\end{equation*}
where we denote $E_{\varepsilon,y}$ by
\begin{equation*}
E_{\varepsilon, y}:=\left\{\varphi \in H_{\varepsilon}:\left\langle\frac{\partial U_{\varepsilon, y^{i}}}{\partial y^{i}}, \varphi\right\rangle_{\varepsilon}=0, i=1, \ldots, N\right\}.
\end{equation*}
We have the following estimates for $U_{\varepsilon,y^{i}}(r)$.
\begin{Lem}\cite[Lemma 2.2]{MR2352959}\label{Lem4.1}
For $i=1, \ldots, N$, in the limit $\varepsilon \rightarrow 0$ the quantities
\begin{equation*}
\left\|U_{ \varepsilon,y^{i}}(r)\right\|^{2}, \quad\left\|\frac{\partial U_{ \varepsilon,y^{i}}(r)}{\partial y^{i}}\right\|^{2}, \quad\left\|\frac{\partial^{2} U_{ \varepsilon,y^{i}}(r)}{\partial^{2} y^{i}}\right\|^{2}
\end{equation*}
have the same order as $\varepsilon^{N}$.
\end{Lem}
\par
We will restrict our argument to the existence of a critical point of $I_{\varepsilon}$ that concentrates, as $\varepsilon$ small enough, near the spheres with radii $\frac{x_{0}}{ \varepsilon}$. Thus we
are looking for a critical point of the form
\begin{equation*}
u_\varepsilon=U_{\varepsilon,y}+\varphi_{\varepsilon}
\end{equation*}
where $\varphi_{\varepsilon} \in E_{\varepsilon, y}$, and $\varepsilon y\rightarrow x_{0},\left\|\varphi_{\varepsilon}\right\|^{2}=o\left(\varepsilon^{N}\right)$ as $\varepsilon \rightarrow 0$. For this we introduce a new functional $J_{\varepsilon}:M_{\varepsilon,\eta} \rightarrow \mathbb{R}$ defined by
\begin{equation*}
J_{\varepsilon}(y, \varphi)=I_{\varepsilon}\left(U_{\varepsilon, y}+\varphi\right), \quad \varphi \in E_{\varepsilon, y}.
\end{equation*}
In fact, we divide the proof of Theorem  \ref{Thm1.3} into two steps:
\begin{itemize}
  \item [{\bf Step1:}] for each $\varepsilon, \delta$ sufficiently small and for each $y \in B_{\delta}(x_0)$, we will find a critical point $\varphi_{\varepsilon, y}$ for $J_{\varepsilon}(y, \cdot)$ (the function $y \mapsto \varphi_{\varepsilon, y}$ also belongs to the class $C^{1}\left(H_{\varepsilon}\right)$ );
  \item [{\bf Step2:}] for each $\varepsilon, \delta$ sufficiently small, we will find a critical point $y_{\varepsilon}$ for the function $j_{\varepsilon}$ : $B_{\delta}(x_0) \rightarrow \mathbb{R}$ induced by
\begin{equation}\label{eq4.1}
y \mapsto j_{\varepsilon}(y) \equiv J\left(y, \varphi_{\varepsilon, y}\right).
\end{equation}
That is, we will find a critical point $y_{\varepsilon}$ in the interior of $B_{\delta}(x_0)$.
\end{itemize}
\par
It is standard to verify that $\left(y_{\varepsilon}, \varphi_{\varepsilon, y_{\varepsilon}}\right)$ is a critical point of $J_{\varepsilon}$ for $\varepsilon$ sufficiently small by the chain rule. This gives a solution $u_{\varepsilon}= U_{\varepsilon, y_{\varepsilon}}+\varphi_{\varepsilon, y_{\varepsilon}}$ to Eq. \eqref{eq1.1} for $\varepsilon$ sufficiently small in virtue of the following lemma.
\begin{Lem}\label{Lem4.2}
There exist $\varepsilon_{0}, \eta_{0}>0$ such that for $\varepsilon \in\left(0, \varepsilon_{0}\right], \eta \in\left(0, \eta_{0}\right]$, and $(y, \varphi) \in M_{\varepsilon, y}$ the following are equivalent:
\begin{itemize}
  \item [$(i)$] $u_{\varepsilon}= U_{\varepsilon, y_{\varepsilon}}+\varphi_{\varepsilon, y_{\varepsilon}}$ is a critical point of $I_{\varepsilon}$ in $H_{\varepsilon}$.
  \item [$(i)$] $(y, \varphi)$ is a critical point of $J_{\varepsilon}$.
\end{itemize}
\end{Lem}
\begin{proof}The arguments follow the approach in \cite{MR2352959,MR1686700}, we give the proof for the nonlocal problem here for completeness. We define for $\eta, \varepsilon>0$
\begin{equation*}
N_{\varepsilon, \eta}=\left\{u \in H_\varepsilon:\|u-U_{\varepsilon,y}\|<\eta \varepsilon^{\frac{1-N}{2}} \text { for some } y \in B_{\delta}(x_0)\right\}.
\end{equation*}
{\bf Claim 1:}
There exists $\eta_{0}, \varepsilon_{0}>0$ such that for $u \in N_{\varepsilon, \eta}$ with $\eta \in\left(0, \eta_{0}\right)$, $\varepsilon \in\left(0, \varepsilon_{0}\right]$, the minimization problem
\begin{equation}\label{eq4.2}
\inf \left\{\|u-U_{\varepsilon,y}\|: y \in D_{\varepsilon, 4 \delta}\right\}
\end{equation}
is achieved in $D_{\varepsilon, 2 \delta}$ and not in $D_{\varepsilon, 4 \delta} \backslash \bar{D}_{\varepsilon, 2 \delta} .$
\par
We argue by contradiction. Assume to the contrary that for some sequences $\varepsilon_{n} \rightarrow 0, \eta_{n} \rightarrow 0, u_{n} \in N_{\varepsilon_{n}, \eta_{n}}$, there exist minimizers $y_{n}$ of \eqref{eq4.2} in $D_{\varepsilon_{n}, 4 \delta} \backslash \bar{D}_{\varepsilon_{n}, 2 \delta} .$ Since $u_{n} \in N_{\varepsilon_{n}, \eta_{n}}$ there exists $\tilde{y}_{n} \in D_{\varepsilon_{n}, \delta}$ with $\left\|u_{n}-U_{\varepsilon,\tilde{y}_{n} }\right\|<\eta_{n} \varepsilon^{\frac{1-N}{2}} .$ It follows that
\begin{equation*}
\varepsilon^{\frac{N-1}{2}}\left\|U_{\varepsilon,\tilde{y}_{n} }-U_{\varepsilon,y_{n} }\right\| \rightarrow 0 \quad \text { as } n \rightarrow \infty.
\end{equation*}
It is easy to check that this implies
\begin{equation*}
\left|y_{n}-\tilde{y}_{n}\right| \rightarrow 0 \quad \text { as } n \rightarrow \infty.
\end{equation*}
 This contradicts the choice of $y_{n}$ and $\tilde{y}_{n}$ and proves Claim 1.
\par
Let $y$ be a minimizer of minimization problem \eqref{eq4.2}, so $y\in D_{\varepsilon, 2 \delta}$. Then
\begin{equation*}
\varphi:=u-U_{\varepsilon,y}
\end{equation*}
satisfies
\begin{equation*}
\left(\varphi, \frac{\partial U_{\varepsilon,y^{i}}}{\partial y^{i}}\right)=0 \quad \text { for } i=1, \ldots, N.
\end{equation*}
{\bf Claim 2:}
For $\eta$ and $\varepsilon$ small enough, the minimization problem \eqref{eq4.2} admits a unique solution.
\par
We proceed by contradiction. Suppose that there exist sequences $\eta_{n}, \varepsilon_{n} \rightarrow 0$ $u_{n} \in N_{\varepsilon_{n}, \eta_{n}}$, such that there are two minimizers $y_{n}$ and $\tilde{y}_{n}$ of \eqref{eq4.2}. By the argument of Claim 1, we have
\begin{equation*}
\left|y_{n}-\tilde{y}_{n}\right| \rightarrow 0 \quad \text { as } n \rightarrow \infty.
\end{equation*}
We show next that $y_{n}=\tilde{y}_{n}$ for sufficiently large $n$. Setting
\begin{equation*}
\varphi_{n}=u_{n}-U_{\varepsilon_{n},y_{n}}, \quad \tilde{\varphi}_{n}=u_{n}-U_{\varepsilon_{n},\tilde{y}_{n}},
\end{equation*}
we have
\begin{equation*}
\left(\varphi_{n}, \frac{\partial  U_{\varepsilon_{n},y^{i}_{n}}}{\partial y^{i}_{n}}\right)=0, \quad\left(\tilde{\varphi}_{n}, \frac{\partial U_{\varepsilon_{n},\tilde{y}^{i}_{n}}}{\partial \bar{y}^{i}_{n}}\right)=0, \quad \text { for } i=1, \ldots, N.
\end{equation*}
hence
\begin{equation}\label{eq4.3}
\sum_{i=1}^{k}\left(U_{\varepsilon_{n},y^{i}_{n}}-U_{\varepsilon_{n},\tilde{y}^{i}_{n}},\frac{\partial  U_{\varepsilon_{n},y^{j}_{n}}}{\partial y^{j}_{n}}\right)=\left(\tilde{\varphi}_{n}, \frac{\partial  U_{\varepsilon_{n},y^{j}_{n}}}{\partial y^{j}_{n}}-\frac{\partial U_{\varepsilon_{n},\tilde{y}^{j}_{n}}}{\partial \bar{y}^{j}_{n}}\right)
\end{equation}
for $j=1, \ldots, N$. Next observe that
\begin{equation*}
\begin{aligned}
&\left(U_{\varepsilon_{n},y^{i}_{n}}-U_{\varepsilon_{n},\tilde{y}^{i}_{n}},\frac{\partial  U_{\varepsilon_{n},y^{j}_{n}}}{\partial y^{j}_{n}}\right) \\
&=-\left(\frac{\partial  U_{\varepsilon_{n},y^{i}_{n}}}{\partial y^{i}_{n}}, \frac{\partial  U_{\varepsilon_{n},y^{j}_{n}}}{\partial y^{j}_{n}}\right)\left(\tilde{y}^{i}_{n}-y^{i}_{n}\right) \\
&+O\left[\left(\frac{\partial^2  U_{\varepsilon_{n},y^{i}_{n}}}{\partial^2 y^{i}_{n}}, \frac{\partial^2  U_{\varepsilon_{n},y^{j}_{n}}}{\partial^2 y^{j}_{n}}\right)\left|\tilde{y}^{i}_{n}-y^{i}_{n}\right|^{2}\right].
\end{aligned}
\end{equation*}
By Lemma \ref{Lem4.1}, the polynomial decay of $U_{\varepsilon,y^{i}}$ and the fact $\left|y^{i}_{n}-y^{i}_{n}\right| \geq \frac{c}{\varepsilon_n}$ (for $i \neq j, c$ independent of $n$ ) we derive the following estimates,
\begin{equation*}
\begin{aligned}
&\left(\frac{\partial  U_{\varepsilon_{n},y^{j}_{n}}}{\partial y^{j}_{n}}, \frac{\partial  U_{\varepsilon_{n},y^{j}_{n}}}{\partial y^{j}_{n}}\right) =C_{j} \varepsilon_{n}^{1-N}+O\left(\varepsilon_{n}^{2-N}\right), \\
&\left(\frac{\partial^2  U_{\varepsilon_{n},y^{j}_{n}}}{\partial^2 y^{j}_{n}}, \frac{\partial  U_{\varepsilon_{n},y^{j}_{n}}}{\partial y^{j}_{n}}\right) =\tilde{C}_{j} \varepsilon_{n}^{1-N}+O\left(\varepsilon_{n}^{2-N}\right), \\
&\left(\frac{\partial  U_{\varepsilon_{n},y^{i}_{n}}}{\partial y^{i}_{n}}, \frac{\partial  U_{\varepsilon_{n},y^{j}_{n}}}{\partial y^{j}_{n}}\right) =O\left(\varepsilon_{n}^{1-N} \frac{1}{1+|\frac{y^i-y^j}{\varepsilon_n}|^{N+2s}}\right), \\
&\left(\frac{\partial^2  U_{\varepsilon_{n},y^{i}_{n}}}{\partial^2 y^{i}_{n}}, \frac{\partial  U_{\varepsilon_{n},y^{j}_{n}}}{\partial y^{j}_{n}}\right) =O\left(\varepsilon_{n}^{1-N} \frac{1}{1+|\frac{y^i-y^j}{\varepsilon_n}|^{N+2s}}\right),
\end{aligned}
\end{equation*}
for $i, j=1, \ldots, N, i \neq j$. Consequently
\begin{equation}\label{eq4.4}
\begin{aligned}
&\left(U_{\varepsilon_{n},y^{i}_{n}}-U_{\varepsilon_{n},\tilde{y}^{i}_{n}},\frac{\partial  U_{\varepsilon_{n},y^{j}_{n}}}{\partial y^{j}_{n}}\right) \\
&=-C_{j} \varepsilon_{n}^{1-N}\left(\tilde{y}^{i}_{n}-y^{i}_{n}\right)+o(1) \sum_{i=1}^{j} \varepsilon_{n}^{1-N}\left(\tilde{y}^{i}_{n}-y^{i}_{n}\right).
\end{aligned}
\end{equation}
Similarly, since $\left\|\varphi_{n}\right\|=o(1) \varepsilon^{(1-N) / 2}$ by assumption we have
\begin{equation}\label{eq4.5}
\begin{aligned}
\left(\tilde{\varphi}_{n}, \frac{\partial  U_{\varepsilon_{n},y^{j}_{n}}}{\partial y^{j}_{n}}-\frac{\partial  U_{\varepsilon_{n},\tilde{y}^{j}_{n}}}{\partial \tilde{y}^{j}_{n}}\right) &=-\left(\tilde{\varphi}_{n}, \frac{\partial^2  U_{\varepsilon_{n},y^{j}_{n}}}{\partial^2 y^{j}_{n}}\right)\left(\tilde{y}^{j}_{n}-y^{j}_{n}\right) \\
&=o(1) \varepsilon^{1-N}\left|\tilde{y}^{j}_{n}-y^{j}_{n}\right|.
\end{aligned}
\end{equation}
Combining \eqref{eq4.3}-\eqref{eq4.5} we see that
\begin{equation*}
\left|\tilde{y}^{j}_{n}-y^{j}_{n}\right|=o(1) \sum_{i=1}^{k}\left|\tilde{y}^{j}_{n}-y^{j}_{n}\right| \quad \text { for } j=1, \ldots, N.
\end{equation*}
Therefore $\bar{y}^{j}_{n}=y^{j}_{n}$ as $n$ large enough, contradicting our assumption.
\par
By Claim 2, there exist $\eta_{0}, \varepsilon_{0}>0$, such that if $\eta \in\left(0, \eta_{0}\right]$ and $\varepsilon \in\left(0, \varepsilon_{0}\right]$, then any $u \in N_{\varepsilon, \eta}$ can be uniquely written in the form
\begin{equation*}
u=U_{\varepsilon,y}+\varphi \quad \text { with }(y, \varphi) \in M_{\varepsilon, \eta}
\end{equation*}
For a given $(y, \varphi) \in M_{\varepsilon, \eta}$, let us define
\begin{equation*}
u=u_{y, \varphi}=U_{\varepsilon,y}+\varphi.
\end{equation*}
Then $(y, \varphi)$ is a critical point of the functional
\begin{equation*}
M_{\varepsilon, \eta} \rightarrow \mathbb{R}, \quad(y, \varphi) \mapsto I_{\varepsilon}\left( U_{\varepsilon,y}+\varphi\right)
\end{equation*}
if and only if
\begin{equation*}
u=U_{\varepsilon,y}+\varphi
\end{equation*}
is a critical point of $I_{\varepsilon}$ in $H_{\varepsilon}$.
\end{proof}
\par
Now, in order to realize {\bf Step 1}, we expand $J_{\varepsilon}(y, \cdot)$ near $\varphi=0$ for each fixed $y$ as follows:
\begin{equation*}
J_{\varepsilon}(y, \varphi)=J_{\varepsilon}(y, 0)+l_{\varepsilon}(\varphi)+\frac{1}{2}\left\langle\mathcal{L}_{\varepsilon} \varphi, \varphi\right\rangle+R_{\varepsilon}(\varphi)
\end{equation*}
where $J_{\varepsilon}(y, 0)=I_{\varepsilon}\left(U_{\varepsilon, y}\right)$, and $l_{\varepsilon}, \mathcal{L}_{\varepsilon}$ and $R_{\varepsilon}$ are defined for $\varphi, \psi \in H_{\varepsilon}$ as follows:
\begin{equation}\label{eq4.6}
\begin{aligned}
l_{\varepsilon}(\varphi) &=\left\langle I_{\varepsilon}^{\prime}\left(U_{\varepsilon, y}\right), \varphi\right\rangle \\
&=\left\langle U_{\varepsilon, y}, \varphi\right\rangle_{\varepsilon}+b\varepsilon^{4s-N}\left(\int_{\mathbb{R}^{N}}\left| (-\Delta)^{\frac{s}{2}}U_{\varepsilon, y}\right|^{2}dx\right) \int_{\mathbb{R}^{N}} (-\Delta)^{\frac{s}{2}} U_{\varepsilon, y} \cdot (-\Delta)^{\frac{s}{2}} \varphi dx-\int_{\mathbb{R}^{N}} U_{\varepsilon, y}^{p} \varphi dx
\end{aligned}
\end{equation}
and $\mathcal{L}_{\varepsilon}: L^{2}\left(\mathbb{R}^{N}\right) \rightarrow L^{2}\left(\mathbb{R}^{N}\right)$ is the bilinear form around $U_{\varepsilon, y}$ defined by
\begin{equation*}
\begin{aligned}
\left\langle\mathcal{L}_{\varepsilon} \varphi, \psi\right\rangle &=\left\langle I_{\varepsilon}^{\prime \prime}\left(U_{\varepsilon, y}\right)[\varphi], \psi\right\rangle \\
&=\langle\varphi, \psi\rangle_{\varepsilon}+ b\varepsilon^{4s-N}\left(\int_{\mathbb{R}^{N}}\left|(-\Delta)^{\frac{s}{2}} U_{\varepsilon, y}\right|^{2}dx\right) \int_{\mathbb{R}^{N}} (-\Delta)^{\frac{s}{2}} \varphi \cdot (-\Delta)^{\frac{s}{2}} \psi dx\\
&+2\varepsilon^{4s-N}b\left(\int_{\mathbb{R}^{N}} (-\Delta)^{\frac{s}{2}} U_{\varepsilon, y} \cdot (-\Delta)^{\frac{s}{2}} \varphi dx\right)\left(\int_{\mathbb{R}^{N}} (-\Delta)^{\frac{s}{2}} U_{\varepsilon, y} \cdot (-\Delta)^{\frac{s}{2}} \psi dx\right)-p \int_{\mathbb{R}^{N}} U_{\varepsilon, y}^{p-1} \varphi \psi dx
\end{aligned}
\end{equation*}
and $R_{\varepsilon}$ denotes the second order reminder term given by
\begin{equation}\label{eq4.7}
R_{\varepsilon}(\varphi)=J_{\varepsilon}(y, \varphi)-J_{\varepsilon}(y, 0)-l_{\varepsilon}(\varphi)-\frac{1}{2}\left\langle\mathcal{L}_{\varepsilon} \varphi, \varphi\right\rangle.
\end{equation}
We remark that $R_{\varepsilon}$ belongs to $C^{2}\left(H_{\varepsilon}\right)$ since so is every term in the right hand side of \eqref{eq4.7}.
In the rest of this section, we consider $l_{\varepsilon}: H_{\varepsilon} \rightarrow \mathbb{R}$ and $R_{\varepsilon}: H_{\varepsilon} \rightarrow \mathbb{R}$ and give some elementary estimates.
\par
We will repeatedly use the following type of Sobolev inequality:
\begin{Lem}\label{Lem4.3}
For any $2 \leq q \leq 2^*_s$, there exists a constant $C>0$ depending only on $N, V, a$ and $q$, but independent of $\varepsilon$, such that
\begin{equation}\label{eq4.8}
\|\varphi\|_{L^{q}\left(\mathbb{R}^{N}\right)} \leq C \varepsilon^{\frac{N}{q}-\frac{N}{2}}\|\varphi\|_{\varepsilon}
\end{equation}
holds for all $\varphi \in H_{\varepsilon} .$
\end{Lem}
\begin{proof}
By setting $\tilde{\varphi}(x)=\varphi(\varepsilon x)$ and using Sobolev inequality, we deduce
\begin{equation*}
\begin{aligned}
\int_{\mathbb{R}^{N}}|\varphi|^{q}dx &=\varepsilon^{N} \int_{\mathbb{R}^{N}}|\tilde{\varphi}|^{q}dx \\
& \leq C_{1} \varepsilon^{N}\left(\int_{\mathbb{R}^{N}}\left(|(-\Delta)^{\frac{s}{2}} \tilde{\varphi}|^{2}+|\tilde{\varphi}|^{2}\right)dx\right)^{q / 2} \\
&=C_{1} \varepsilon^{N-\frac{Nq}{2}}\left(\int_{\mathbb{R}^{N}}\left(\varepsilon^{2s}|(-\Delta)^{\frac{s}{2}} \varphi|^{2}+|\varphi|^{2}\right)dx\right)^{q / 2} \\
& \leq C_{2} \varepsilon^{N-\frac{Nq}{2}}\|\varphi\|_{\varepsilon}^{q}
\end{aligned}
\end{equation*}
where $C_{1}$ is the best constant for the fractional Sobolev embedding $H^{s}\left(\mathbb{R}^{N}\right) \subset L^{q}\left(\mathbb{R}^{N}\right)$, and $C_{2}>0$ depends only on $n, a, q$ and $V$.
\end{proof}

\begin{Lem}\cite{MR2595734}\label{Lem4.6}
For any constant $0<\sigma \leq \min \{\alpha, \beta\}$, there is a constant $C>0$, such that
\begin{equation*}
\frac{1}{\left(1+\left|y-x^{i}\right|\right)^{\alpha}} \frac{1}{\left(1+\left|y-x^{j}\right|\right)^{\beta}} \leq \frac{C}{\left|x^{i}-x^{j}\right|^{\sigma}}\left(\frac{1}{\left(1+\left|y-x^{i}\right|\right)^{\alpha+\beta-\sigma}}+\frac{1}{\left(1+\left|y-x^{j}\right|\right)^{\alpha+\beta-\sigma}}\right)
\end{equation*}
where $\alpha, \beta>0$ are two constants.
\end{Lem}

\begin{Lem}\label{Lem4.4}
Assume that $V$ satisfies $(V_1)$ and $(V_2)$. Then, there exists a constant $C>0$, independent of $\varepsilon$, such that for any $y \in B_{1}(0)$, there holds
\begin{equation*}
\left|l_{\varepsilon}(\varphi)\right| \leq C \varepsilon^{\frac{N}{2}}\left(\varepsilon^{\alpha}+(|V(y)-V(x_0)|)\right)\|\varphi\|_{\varepsilon}
\end{equation*}
for $\varphi \in H_{\varepsilon}$. Here $\alpha$ denotes the order of the H\"{o}lder continuity of $V$ in $B_{r_0}(0)$.
\end{Lem}
\begin{proof}Since $U$ is the unique positive solution of Eq. \eqref{eq1.2} and the definition of $U_{\varepsilon,y}$, we can write
\begin{equation*}
l_{\varepsilon}(\varphi)=\int_{\mathbb{R}^{N}}\left(V(x)-V\left(x_0\right)\right) U_{\varepsilon, y} \varphi dx.
\end{equation*}
Using the conditions imposed on $V(x)$ and $\alpha<\frac{N+4 s}{2}$, we have
\begin{equation}\label{eq4.9}
\begin{aligned}
&\left| \int_{\mathbb{R}^{N}}\left(V(x)-V\left(y^i\right)\right) U_{\varepsilon, y} \varphi dx \right| \\
& \leq\left(\int_{\mathbb{R}^{N}}\left(\left(V(x)-V\left(y\right)\right) U_{\varepsilon, y}\right)^{2}d x\right)^{\frac{1}{2}}\|\varphi\|_{\varepsilon} \\
& \leq\left(\int_{\mathbb{R}^{N}}\left(V\left(\varepsilon x+y\right)-V\left(y\right) U\right)^{2} \varepsilon^{N}d y\right)^{\frac{1}{2}}\|\varphi\|_{\varepsilon} \\
& \leq \varepsilon^{\frac{N}{2}+\alpha}\left(\int_{\mathbb{R}^{N}}|y|^{2 \alpha} U^{2}d y\right)^{\frac{1}{2}}\|\varphi\|_{\varepsilon} \\
& \leq C \varepsilon^{\frac{N}{2}+\alpha}\left(C+\left(\int_{\mathbb{R}^{N} \backslash B_{R}(0)} \frac{|y|^{2 \alpha}}{1+|y|^{2(N+2 s)}}d y\right)^{\frac{1}{2}}\|\varphi\|_{\varepsilon}\right.\\
& \leq C \varepsilon^{\frac{N}{2}+\alpha}\|\varphi\|_{\varepsilon}
\end{aligned}
\end{equation}
and
\begin{equation}\label{eq4.10}
\begin{aligned}
&\left|\int_{\mathbb{R}^{N}}\left(V\left(y\right)-V\left(x_{0}\right)\right) U_{\varepsilon, y} \varphi dx\right| \\
& \leq C \varepsilon^{\frac{N}{2}}\left|V\left(y\right)-V\left(x_{0}\right)\right|\left(\int_{\mathbb{R}^{N}} U^{2}d x\right)^{\frac{1}{2}}\|\varphi\|_{\varepsilon} \\
& \leq C \varepsilon^{\frac{N}{2}}\left|V\left(y\right)-V\left(x_{0}\right)\right|\|\varphi\|_{\varepsilon}.
\end{aligned}
\end{equation}
Combining \eqref{eq4.9} and \eqref{eq4.10}, we obtain
\begin{equation}\label{eq4.11}
\begin{aligned}
&\left|\int_{\mathbb{R}^{N}}\left(V(x)-V\left(x_{0}\right)\right) U_{\varepsilon, y} \varphi dx \right| \\
&\quad=\left|\int_{\mathbb{R}^{N}}\left(V(x)-V\left(y\right)\right) U_{\varepsilon, y} \varphi d x+\int_{\mathbb{R}^{N}} \left(V\left(y\right)-V\left(x_{0}\right)\right) U_{\varepsilon, y} \varphi d x\right| \\
&\quad \leq C \varepsilon^{\frac{N}{2}+\alpha}\|\varphi\|_{\varepsilon}+C \varepsilon^{\frac{N}{2}} \left|V\left(y\right)-V\left(x^{0}\right)\right|\|\varphi\|_{\varepsilon}.
\end{aligned}
\end{equation}
It follows form \eqref{eq4.11} that the result has been proved.
\end{proof}
\par
Next we give estimates for $R_{\varepsilon}$ and its derivatives $R_{\varepsilon}^{(i)}$ for $i=1,2$.
\begin{Lem}\label{Lem4.5}
There exists a constant $C>0$, independent of $\varepsilon$ and $b$, such that for $i \in\{0,1,2\}$, there hold
\begin{equation*}
\left\|R_{\varepsilon}^{(i)}(\varphi)\right\| \leq C \varepsilon^{-\frac{N(p-1)}{2}}\|\varphi\|_{\varepsilon}^{p+1-i}+C(b+1) \varepsilon^{-\frac{N}{2}}\left(1+\varepsilon^{-\frac{N}{2}}\|\varphi\|_{\varepsilon}\right)\|\varphi\|_{\varepsilon}^{N-i}
\end{equation*}
for all $\varphi \in H_{\varepsilon}$.
\end{Lem}
\begin{proof}
By the definition of $R_{\varepsilon}$ in \eqref{eq4.7}, we have
\begin{equation*}
R_{\varepsilon}(\varphi)=A_{1}(\varphi)-A_{2}(\varphi)
\end{equation*}
where
\begin{equation*}
A_{1}(\varphi)=\frac{b \varepsilon^{4s-N}}{4}\left(\left(\int_{\mathbb{R}^{N}}|(-\Delta)^{\frac{s}{2}} \varphi|^{2}dx \right)^{2}+4 \int_{\mathbb{R}^{N}}|(-\Delta)^{\frac{s}{2}} \varphi|^{2}dx \int_{\mathbb{R}^{N}} (-\Delta)^{\frac{s}{2}} U_{\varepsilon, y} \cdot (-\Delta)^{\frac{s}{2}} \varphi dx \right)
\end{equation*}
and
\begin{equation*}
A_{2}(\varphi)=\frac{1}{p+1} \int_{\mathbb{R}^{N}}\left(\left(U_{\varepsilon, y}+\varphi\right)_{+}^{p+1}-U_{\varepsilon, y}^{p+1}-(p+1) U_{\varepsilon, y}^{p} \varphi-\frac{p(p+1)}{2} U_{\varepsilon, y}^{p-1} \varphi^{2}\right)dx.
\end{equation*}
Use $R_{\varepsilon}^{(i)}$ to denote the $i$ th derivative of $R_{\varepsilon}$, and also use similar notations for $A_{1}$ and $A_{2} .$ By direct computations, we deduce that, for any $\varphi, \psi \in H_{\varepsilon}$,
\begin{equation*}
\left\langle R_{\varepsilon}^{(1)}(\varphi), \psi\right\rangle=\left\langle A_{1}^{(1)}(\varphi), \psi\right\rangle-\left\langle A_{2}^{(1)}(\varphi), \psi\right\rangle
\end{equation*}
where
\begin{equation*}
\begin{aligned}
\left\langle A_{1}^{(1)}(\varphi), \psi\right\rangle=& b \varepsilon^{4s-N}\left(\int_{\mathbb{R}^{N}}|(-\Delta)^{\frac{s}{2}} \varphi|^{2}dx \int_{\mathbb{R}^{N}} (-\Delta)^{\frac{s}{2}} \varphi \cdot (-\Delta)^{\frac{s}{2}} \psi dx\right.\\
&\left.+\int_{\mathbb{R}^{N}}|(-\Delta)^{\frac{s}{2}} \varphi|^{2}dx \int_{\mathbb{R}^{N}} (-\Delta)^{\frac{s}{2}} U_{\varepsilon, y} \cdot (-\Delta)^{\frac{s}{2}} \psi dx\right) \\
&+2 b \varepsilon^{4s-N} \int_{\mathbb{R}^{N}} (-\Delta)^{\frac{s}{2}} U_{\varepsilon, y} \cdot (-\Delta)^{\frac{s}{2}} \varphi dx \int_{\mathbb{R}^{N}} (-\Delta)^{\frac{s}{2}} \varphi \cdot (-\Delta)^{\frac{s}{2}} \psi dx
\end{aligned}
\end{equation*}
and
\begin{equation*}
\left\langle A_{2}^{(1)}(\varphi), \psi\right\rangle=\int_{\mathbb{R}^{N}}\left(\left(U_{\varepsilon, y}+\varphi\right)_{+}^{p} \psi-U_{\varepsilon, y}^{p} \psi-p U_{\varepsilon, y}^{p-1} \varphi \psi\right)dx.
\end{equation*}
We also deduce, for any $\varphi, \psi, \xi \in H_{\varepsilon}$, that
\begin{equation*}
\left\langle R_{\varepsilon}^{(2)}(\varphi)[\psi], \xi\right\rangle=\left\langle A_{1}^{(2)}(\varphi)[\psi], \xi\right\rangle-\left\langle A_{2}^{(2)}(\varphi)[\psi], \xi\right\rangle
\end{equation*}
where
\begin{equation*}
\begin{aligned}
\left\langle A_{1}^{(2)}(\varphi)[\psi], \xi\right\rangle=& b \varepsilon^{4s-N}\left(2 \int_{\mathbb{R}^{N}} (-\Delta)^{\frac{s}{2}} \varphi \cdot (-\Delta)^{\frac{s}{2}} \psi dx \int_{\mathbb{R}^{N}} (-\Delta)^{\frac{s}{2}} \varphi \cdot (-\Delta)^{\frac{s}{2}} \xi dx\right.\\
&\left.+\int_{\mathbb{R}^{N}}|(-\Delta)^{\frac{s}{2}} \varphi|^{2} dx \int_{\mathbb{R}^{N}} (-\Delta)^{\frac{s}{2}} \xi \cdot (-\Delta)^{\frac{s}{2}} \psi dx\right) \\
&+2 b \varepsilon^{4s-N}\left(\int_{\mathbb{R}^{N}} (-\Delta)^{\frac{s}{2}} \varphi \cdot (-\Delta)^{\frac{s}{2}} \psi dx \int_{\mathbb{R}^{N}} (-\Delta)^{\frac{s}{2}} U_{\varepsilon, y} \cdot (-\Delta)^{\frac{s}{2}} \xi dx \right)\\
&+2 b \varepsilon^{4s-N}\left(\int_{\mathbb{R}^{N}} (-\Delta)^{\frac{s}{2}} U_{\varepsilon, y} \cdot (-\Delta)^{\frac{s}{2}} \psi dx \int_{\mathbb{R}^{N}} (-\Delta)^{\frac{s}{2}} \varphi \cdot (-\Delta)^{\frac{s}{2}} \xi dx\right) \\
&+2 b \varepsilon^{4s-N} \int_{\mathbb{R}^{N}} (-\Delta)^{\frac{s}{2}} U_{\varepsilon, y} \cdot (-\Delta)^{\frac{s}{2}} \varphi dx \int_{\mathbb{R}^{N}} (-\Delta)^{\frac{s}{2}} \xi \cdot (-\Delta)^{\frac{s}{2}} \psi dx
\end{aligned}
\end{equation*}
and
\begin{equation*}
\left\langle A_{2}^{(2)}(\varphi)[\psi], \xi\right\rangle=\int_{\mathbb{R}^{N}}\left(p\left(U_{\varepsilon, y}+\varphi\right)_{+}^{p-1} \psi \xi-p U_{\varepsilon, y}^{p-1} \psi \xi\right)dx.
\end{equation*}
\par
First, we estimate $A_{1}(\varphi), A_{1}^{1}(\varphi)$ and $A_{1}^{2}(\varphi)$. Notice that
\begin{equation*}
\left\|(-\Delta)^{\frac{s}{2}} U_{\varepsilon, y}\right\|_{L^{2}\left(\mathbb{R}^{N}\right)}=C_{0} \varepsilon^{\frac{N-2s}{2}}
\end{equation*}
with $C_{0}=\|(-\Delta)^{\frac{s}{2}} U\|_{L^{2}\left(\mathbb{R}^{N}\right)}$, and that
\begin{equation*}
\|(-\Delta)^{\frac{s}{2}} \varphi\|_{L^{2}\left(\mathbb{R}^{N}\right)} \leq C_{1} \varepsilon^{2s-N}\|\varphi\|_{\varepsilon}, \quad \varphi \in H_{\varepsilon}
\end{equation*}
holds for some $C_{1}>0$ independent of $\varepsilon$. Combining above two estimates together with H\"{o}lder's inequality yields
\begin{equation*}
\int_{\mathbb{R}^{N}}|(-\Delta)^{\frac{s}{2}} \varphi \cdot (-\Delta)^{\frac{s}{2}} \psi| dx \int_{\mathbb{R}^{N}}\left|(-\Delta)^{\frac{s}{2}} U_{\varepsilon, y} \cdot (-\Delta)^{\frac{s}{2}} \xi\right| dx\leq C \varepsilon^{-\frac{N+2s}{2}}
\end{equation*}
and that
\begin{equation*}
\int_{\mathbb{R}^{N}}|(-\Delta)^{\frac{s}{2}} \varphi \cdot (-\Delta)^{\frac{s}{2}} \psi|dx \int_{\mathbb{R}^{N}}|(-\Delta)^{\frac{s}{2}} \eta \cdot (-\Delta)^{\frac{s}{2}} \xi| dx\leq C \varepsilon^{-2N}
\end{equation*}
for all $\varphi, \psi, \eta, \xi \in H_{\varepsilon} .$ These estimates imply that
\begin{equation*}
\left|A_{1}^{(i)}(\varphi)\right| \leq C b \varepsilon^{-\frac{N}{2}}\left(1+\varepsilon^{-\frac{N}{2}}\|\varphi\|_{\varepsilon}\right)\|\varphi\|_{\varepsilon}^{N-i}
\end{equation*}
for some constant $C>0$ independent of $\varepsilon$.
\par
Next we estimate $A_{2}^{(i)}(\varphi)$ (the $i$ th derivative of $\left.A_{2}(\varphi)\right)$ for $i=0,1,2$. We consider the case $1<p \leq 2$ first.
\par
On one hand, if $1<p<2$, from \eqref{eq4.8}, we find
\begin{equation*}
\begin{aligned}
\left|A_{2}(\varphi)\right| & \leq C\left|\int_{\mathbb{R}^{N}} \varphi^{(p+1)}d x\right| \leq C \varepsilon^{\frac{1-p}{2} N}\|\varphi\|_{\varepsilon}^{(p+1)}, \\
\left|\left\langle A_{2}^{1}(\varphi), \psi\right\rangle\right| & \leq C\left|\int_{\mathbb{R}^{N}} \varphi^{p} \psi d x\right| \\
& \leq C\left(\int_{\mathbb{R}^{N}}|\varphi|^{p+1}d x\right)^{\frac{p}{p+1}}\left(\int_{\mathbb{R}^{N}}|\psi|^{p+1}d x\right)^{\frac{1}{p+1}} \\
& \leq C\left(\varepsilon^{\frac{1-p}{2} N}\|\varphi\|_{\varepsilon}^{p+1}\right)^{\frac{p}{p+1}}\left(\varepsilon^{\frac{1-p}{2} N}\|\psi\|_{\varepsilon}^{p+1}\right)^{\frac{1}{p+1}} \\
& \leq C \varepsilon^{\frac{1-p}{2} N}\|\varphi\|_{\varepsilon}^{p}\|\psi\|_{\varepsilon}
\end{aligned}
\end{equation*}
and
\begin{equation*}
\begin{aligned}
\left|\left\langle A_{2}^{2}(\varphi)(\psi, \xi)\right\rangle\right| & \leq C\left|\int_{\mathbb{R}^{N}} \varphi^{p-1} \psi \xi d x\right| \\
& \leq C\left(\int_{\mathbb{R}^{N}}|\varphi|^{(p-1) \frac{p+1}{p-1}}d x\right)^{\frac{p-1}{p+1}}\left(\int_{\mathbb{R}^{N}}|\psi|^{p+1}d x\right)^{\frac{1}{p+1}}\left(\int_{\mathbb{R}^{N}}|\xi|^{p+1}d x\right)^{\frac{1}{p+1}} \\
& \leq C\left(\varepsilon^{\frac{1-p}{2} N}\|\varphi\|_{\varepsilon}^{p+1}\right)^{\frac{p-1}{p+1}}\left(\varepsilon^{\frac{1-p}{2} N}\|\psi\|_{\varepsilon}^{p+1}\right)^{\frac{1}{p+1}}\left(\varepsilon^{\frac{1-p}{2} N}\|\xi\|_{\varepsilon}^{p+1}\right)^{\frac{1}{p+1}} \\
& \leq C \varepsilon^{\frac{1-p}{2} N}\|\varphi\|_{\varepsilon}^{p-1}\|\psi\|_{\varepsilon}\|\xi\|_{\varepsilon}
\end{aligned}
\end{equation*}
On the other hand, for the case $p>2$, using \eqref{eq4.8}, we also can obtain that
\begin{equation*}
\begin{aligned}
&\left|R_{\varepsilon}(\varphi)\right| \leq C \int_{\mathbb{R}^{N}} U_{\varepsilon, y}^{p-2} \varphi^{3} \leq C \varepsilon^{-\frac{N}{2}}\|\varphi\|_{\varepsilon}^{3}, \\
&\left|\left\langle R_{\varepsilon}^{\prime}(\varphi), \psi\right\rangle\right| \leq C \varepsilon^{-\frac{N}{2}}\|\varphi\|_{\varepsilon}^{2}\|\psi\|_{\varepsilon},
\end{aligned}
\end{equation*}
and
\begin{equation*}
\left|\left\langle R_{\varepsilon}^{\prime \prime}(\varphi)(\psi, \xi)\right\rangle\right| \leq C \varepsilon^{-\frac{N}{2} N}\|\varphi\|_{\varepsilon}\|\psi\|_{\varepsilon}\|\xi\|_{\varepsilon}.
\end{equation*}
So the results follow.
\end{proof}
\par
Now we will give the energy expansion for the approximate solutions.
\begin{Lem}\label{Lem4.7}
Assume that $V$ satisfies $(V 1)$ and $(V 2)$. Then, for $\varepsilon>0$ sufficiently small, there is a small constant $\tau>0$ and $C>0$ such that,
\begin{equation*}
\begin{aligned}
I_{\varepsilon}\left(U_{\varepsilon, y}\right)=& A \varepsilon^{N}+B \varepsilon^{N}\left(\left(V\left(y\right)-V\left(x_0\right)\right)\right)+O(\varepsilon^{N+\alpha})
\end{aligned}
\end{equation*}
where
\begin{equation*}
A=\frac{1}{2} \int_{\mathbb{R}^{N}}\left(a|(-\Delta)^{\frac{s}{2}} U|^{2}+U^{2}\right)dx+\frac{b}{4}\left(\int_{\mathbb{R}^{N}}|(-\Delta)^{\frac{s}{2}} U|^{2}dx\right)^{2}-\frac{1}{p+1} \int_{\mathbb{R}^{N}} U^{p+1}dx,
\end{equation*}
and
\begin{equation*}
B=\frac{1}{2} \int_{\mathbb{R}^{N}} U^{2}dx.
\end{equation*}
\end{Lem}
\begin{proof}
By direct computation, we can know
\begin{equation}\label{eq4.14}
\begin{aligned}
I_{\varepsilon}\left(U_{\varepsilon, y}\right)=&\frac{1}{2} \int_{\mathbb{R}^{N}}\left(\varepsilon^{2} a\left|(-\Delta)^{\frac{s}{2}} U_{\varepsilon, y}\right|^{2}+V(x) U_{\varepsilon, y}^{2}\right)dx+\frac{\varepsilon^{4s-N} b}{4}\left(\int_{\mathbb{R}^{N}}\left|(-\Delta)^{\frac{s}{2}} U_{\varepsilon, y}\right|^{2}dx\right)^{2}\\
&\quad-\frac{1}{p+1} \int_{\mathbb{R}^{N}} U_{\varepsilon, y}^{p+1}dx \\
=&A\varepsilon^N+ \frac{1}{2} \int_{\mathbb{R}^{N}}\left(V(x)-V\left(x_0\right)\right) U_{\varepsilon, y}^{2}dx.
\end{aligned}
\end{equation}
Now, we discuss the last term in the right hand of \eqref{eq4.14}. Firstly, by direct calculation, we have
\begin{equation}\label{eq4.15}
\begin{aligned}
\int_{\mathbb{R}^{N}}\left(V(x)-V(x_0)\right) U_{\varepsilon, y}^{2} d x
&= \int_{\mathbb{R}^{N}}\left(V(x)-V\left(y\right)+V\left(y\right)-V(x_0)\right) U_{\varepsilon, y}^{2} d x \\
&=\varepsilon^{N} \int_{\mathbb{R}^{N}} U^{2}\left(V(y)-V\left(x_0\right)\right)+\int_{\mathbb{R}^{N}} \left(V(x)-V\left(y\right)\right) U_{\varepsilon, y}^{2} d x. \\
\end{aligned}
\end{equation}
Secondly, we have
\begin{equation}\label{eq4.16}
\begin{aligned}
\int_{\mathbb{R}^{N}}\left(V(x)-V\left(y\right)\right) U_{\varepsilon, y}^{2} d x &=\varepsilon^{N} \int_{\mathbb{R}^{N}}\left(V\left(\varepsilon z+y\right)-V\left(y\right)\right) U^{2} dz \\
& \leq \varepsilon^{N} \int_{\mathbb{R}^{N}} \frac{|\varepsilon z|^{\alpha}}{1+|z|^{2(N+2 s)}}d z \\
& \leq \varepsilon^{N+\alpha}\left(C+\int_{\mathbb{R}^{N} \backslash B_{R}(0)} \frac{1}{|z|^{(2 N+4 s-\alpha)}}dz\right) \\
& \leq C \varepsilon^{N+\alpha}.
\end{aligned}
\end{equation}
So, combining \eqref{eq4.15}-\eqref{eq4.14}, we obtain
\begin{equation}\label{eq4.19}
\int_{\mathbb{R}^{N}}\left(V(x)-V(x_0)\right) U_{\varepsilon, y}^{2} d x=\varepsilon^{N} \left(V(y)-V\left(x_0\right)\right)+O\left(\varepsilon^{N+\alpha}\right).
\end{equation}
Combining above estimates gives the desired estimates. The proof is complete.
\end{proof}

\section{Semiclassical solutions for the fractional Kirchhoff equation}

\subsection{Finite dimensional reduction}
In this subsection we complete {\bf Step 1} for the Lyapunov-Schmidt reduction method as in Section $4 .$ We first consider the operator $\mathcal{L}_{\varepsilon}$,
\begin{equation*}
\begin{aligned}
\left\langle\mathcal{L}_{\varepsilon} \varphi, \psi\right\rangle&=\langle\varphi, \psi\rangle_{\varepsilon}+\varepsilon^{4s-N} b\int_{\mathbb{R}^{N}}\left| (-\Delta)^{\frac{s}{2}} U_{\varepsilon, y}\right|^{2}dx \int_{\mathbb{R}^{N}} (-\Delta)^{\frac{s}{2}} \varphi \cdot (-\Delta)^{\frac{s}{2}} \psi dx\\
&\quad+2 \varepsilon^{4s-N} b\left(\int_{\mathbb{R}^{N}} (-\Delta)^{\frac{s}{2}} U_{\varepsilon, y} \cdot (-\Delta)^{\frac{s}{2}} \varphi dx \right)\left(\int_{\mathbb{R}^{N}} (-\Delta)^{\frac{s}{2}} U_{\varepsilon, y} \cdot (-\Delta)^{\frac{s}{2}} \psi dx\right)-p \int_{\mathbb{R}^{N}} U_{\varepsilon, y}^{p-1} \varphi \psi dx
\end{aligned}
\end{equation*}
for $\varphi, \psi \in H_{\varepsilon} .$ The following result shows that $\mathcal{L}_{\varepsilon}$ is invertible when restricted on $E_{\varepsilon, y}$

\begin{Lem}\label{Lem5.1}
There exist $\varepsilon_{1}>0, \delta_{1}>0$ and $\rho>0$ sufficiently small, such that for every $\varepsilon \in\left(0, \varepsilon_{1}\right), \delta \in\left(0, \delta_{1}\right)$, there holds
\begin{equation*}
\left\|\mathcal{L}_{\varepsilon} \varphi\right\|_{\varepsilon} \geq \rho\|\varphi\|_{\varepsilon}, \quad \forall \varphi \in E_{\varepsilon, y}
\end{equation*}
uniformly with respect to $y \in B_{\delta}(x_0)$.
\end{Lem}
\begin{proof}We use a contradiction argument. Suppose that there exist $\varepsilon_{n}, \delta_{n} \rightarrow 0, y_{n} \in B_{\delta}(x_0)$ and $\varphi_{n} \in E_{\varepsilon_{n}, y_{n}}$ satisfying
\begin{equation}\label{eq5.1}
\left\langle\mathcal{L}_{\varepsilon_{n}} \varphi_{n}, g\right\rangle \leq \frac{1}{n}\left\|\varphi_{n}\right\|_{\varepsilon_{n}}\|g\|_{\varepsilon_{n}}, \quad \forall g \in E_{\varepsilon_{n}, y_{n}}.
\end{equation}
Since this inequality is homogeneous with respect to $\varphi_{n}$, we can assume that
\begin{equation*}
\left\|\varphi_{n}\right\|_{\varepsilon_{n}}^{2}=\varepsilon_{n}^{N} \quad \text { for all } n.
\end{equation*}
Denote $\tilde{\varphi}_{n}(x)=\varphi_{n}\left(\varepsilon_{n} x+y_{n}\right) .$ Then
\begin{equation*}
\int_{\mathbb{R}^{N}}\left(a\left|(-\Delta)^{\frac{s}{2}} \tilde{\varphi}_{n}\right|^{2}+V\left(\varepsilon_{n} x+y_{n}\right) \tilde{\varphi}_{n}^{2}\right)dx=1.
\end{equation*}
As $V$ is bounded and $\inf\limits_{\mathbb{R}^{N}} V>0$, we infer that $\left\{\tilde{\varphi}_{n}\right\}$ is a bounded sequence in $H_{\varepsilon}$. Hence, up to a subsequence, we may assume that
\begin{equation*}
\begin{aligned}
\tilde{\varphi}_{n} \rightarrow \varphi & \text { in } H_{\varepsilon}, \\
\tilde{\varphi}_{n} \rightarrow \varphi & \text { in } L_{\text {loc }}^{p+1}\left(\mathbb{R}^{N}\right), \\
\tilde{\varphi}_{n} \rightarrow \varphi & \text { a.e. in } \mathbb{R}^{N},
\end{aligned}
\end{equation*}
for some $\varphi \in H_{\varepsilon}$. We will prove that $\varphi \equiv 0$.
\par
First we prove that $\varphi=\sum\limits_{l=1}^{N} c^{l} \partial_{x_{l}} U$ for some $c^{l} \in \mathbb{R}$. To this end, let $\tilde{E}_{n}=\left\{\tilde{g} \in H_{\varepsilon}: \tilde{g}_{\varepsilon_{n}, y_{n}} \in\right.$ $\left.E_{\varepsilon_{n}, y_{n}}\right\}$, that is,
\begin{equation*}
\tilde{E}_{n}=\left\{\tilde{g} \in H_{\varepsilon}: \int_{\mathbb{R}^{N}}\left(a (-\Delta)^{\frac{s}{2}} \partial_{x_{i}} U \cdot (-\Delta)^{\frac{s}{2}} \tilde{g}+V\left(\varepsilon_{n} x+y_{n}\right) \partial_{x_{i}} U \tilde{g}\right)dx=0 \text { for } i=1,2,\cdots,N\right\}.
\end{equation*}
For convenience, denote at the moment
\begin{equation*}
\langle u, v\rangle_{*, n}=\int_{\mathbb{R}^{N}}\left(a (-\Delta)^{\frac{s}{2}} u \cdot (-\Delta)^{\frac{s}{2}} v+V\left(\varepsilon_{n} x+y_{n}\right) u v\right)dx \quad \text { and } \quad\|u\|_{*, n}^{2}=\langle u, u\rangle_{*, n}
\end{equation*}
Then \eqref{eq5.1} can be rewritten in terms of $\tilde{\varphi}_{n}$ as follows:
\begin{equation}\label{eq5.2}
\begin{aligned}
&\left\langle\tilde{\varphi}_{n}, \tilde{g}\right\rangle_{*, n}+b \int_{\mathbb{R}^{N}}|(-\Delta)^{\frac{s}{2}} U|^{2} dx \int_{\mathbb{R}^{N}} (-\Delta)^{\frac{s}{2}} \tilde{\varphi}_{n} \cdot (-\Delta)^{\frac{s}{2}} \tilde{g}dx \\
+& 2 b \int_{\mathbb{R}^{N}} (-\Delta)^{\frac{s}{2}} U \cdot (-\Delta)^{\frac{s}{2}} \tilde{\varphi}_{n}dx \int_{\mathbb{R}^{N}} (-\Delta)^{\frac{s}{2}} U \cdot (-\Delta)^{\frac{s}{2}} \tilde{g}dx-p \int_{\mathbb{R}^{N}} U^{p-1} \tilde{\varphi}_{n} \tilde{g}dx \\
\leq & n^{-1}\left\|\tilde{g}_{n}\right\|_{*, n}
\end{aligned}
\end{equation}
where $\tilde{g}_{n}(x)=g\left(\varepsilon_{n} x+y_{n}\right) \in \tilde{E}_{n}$.
\par
Now, for any $g \in C_{0}^{\infty}\left(\mathbb{R}^{N}\right)$, define $a_{n}^{l} \in \mathbb{R}(1 \leq l \leq N)$ by
\begin{equation*}
a_{n}^{l}=\frac{\left\langle\partial_{x I} U, g\right\rangle_{*, n}}{\left\|\partial_{x l} U\right\|_{*, n}^{2}}
\end{equation*}
and let $\tilde{g}_{n}=g-\sum\limits_{l=1}^{N} a_{n}^{l} \partial_{x_{l}} U$. Note that
\begin{equation*}
\left\|\partial_{x_{l}} U\right\|_{*, n}^{2} \rightarrow \int_{\mathbb{R}^{N}}\left(a\left|(-\Delta)^{\frac{s}{2}} \partial_{x_{l}} U\right|^{2}+\left(\partial_{x_{l}} U\right)^{2}\right)dx>0
\end{equation*}
and for $l \neq j$
\begin{equation*}
\left\langle\partial_{x_{l}} U, \partial_{x_{j}} U\right\rangle_{*, n}=\int_{\mathbb{R}^{N}} V\left(\varepsilon_{n} x+y_{n}\right) \partial_{x_{l}} U \partial_{x_{j}} Udx \rightarrow \int_{\mathbb{R}^{N}} \partial_{x_{l}} U \partial_{x_{j}} Udx=0
\end{equation*}
Hence the dominated convergence theorem implies that
\begin{equation*}
a_{n}^{l} \rightarrow a^{l}=\frac{\int_{\mathbb{R}^{N}}\left(a (-\Delta)^{\frac{s}{2}} \partial_{x_{l}} U \cdot (-\Delta)^{\frac{s}{2}} g+\partial_{x_{l}} U g\right)dx}{ \int_{\mathbb{R}^{N}}\left(a\left|(-\Delta)^{\frac{s}{2}} \partial_{x_{l}} U\right|^{2}+\left(\partial_{x_{l}} U\right)^{2}\right)dx}
\end{equation*}
and
\begin{equation*}
\left\langle\partial_{x_{l}} U, \tilde{g}_{n}\right\rangle_{*, n} \rightarrow 0
\end{equation*}
as $n \rightarrow \infty$. Moreover, we infer that
\begin{equation*}
\left\|\tilde{g}_{n}\right\|_{*, n}=O(1)
\end{equation*}
Now substituting $\tilde{g}_{n}$ into \eqref{eq5.2} and letting $n \rightarrow \infty$, we find that
\begin{equation*}
\langle\mathcal{L}_+ \varphi, g\rangle-\sum_{l=1}^{N} a^{I}\left\langle\mathcal{L}_+ \varphi, \partial_{x_{l}} U\right\rangle=0
\end{equation*}
where $\mathcal{L}_+$ is defined as in Theorem \ref{Thm1.2}. Since $U_{x l} \in \operatorname{Ker} \mathcal{L}_+$ by Theorem \ref{Thm1.2}, we have $\left\langle\mathcal{L}_+ \varphi, \partial_{x_{i}} U\right\rangle=0$. Thus
\begin{equation*}
\langle\mathcal{L}_+ \varphi, g\rangle=0, \quad \forall g \in C_{0}^{\infty}\left(\mathbb{R}^{n}\right).
\end{equation*}
This implies that $\varphi \in \operatorname{Ker} \mathcal{L}_+$. Applying Theorem \ref{Thm1.2} again gives $c^{l} \in \mathbb{R}$ ( $1 \leq l \leq N$ ) such that
\begin{equation*}
\varphi=\sum_{l=1}^{N} c^{l} \partial_{x_{l}} U.
\end{equation*}
Next we prove $\varphi \equiv 0$. Note that $\tilde{\varphi}_{n} \in \tilde{E}_{n}$, that is,
\begin{equation*}
\int_{\mathbb{R}^{N}}\left(a (-\Delta)^{\frac{s}{2}} \tilde{\varphi}_{n} \cdot (-\Delta)^{\frac{s}{2}} \partial_{x l} U+V\left(\varepsilon_{n} x+y_{n}\right) \tilde{\varphi}_{n} \partial_{x l} U\right)dx=0
\end{equation*}
for each $l=1,2\cdots,N$. By sending $n \rightarrow \infty$, we derive
\begin{equation*}
c^{l} \int_{\mathbb{R}^{N}}\left(a\left|(-\Delta)^{\frac{s}{2}} \partial_{x_{l}} U\right|^{2}+\left(\partial_{x_{l}} U\right)\right)dx=0
\end{equation*}
which implies $c^{l}=0$. Hence
\begin{equation*}
\varphi \equiv 0 \quad \text { in } \mathbb{R}^{N}.
\end{equation*}
Now we can complete the proof. We have proved that $\tilde{\varphi}_{n} \rightarrow 0$ in $H_{\varepsilon}$ and $\tilde{\varphi}_{n} \rightarrow 0$ in $L_{\text {loc }}^{p+1}\left(\mathbb{R}^{N}\right)$. As a result we obtain
\begin{equation*}
\begin{aligned}
p \int_{\mathbb{R}^{N}} U_{\varepsilon_{n}, y_{n}}^{p-1} \varphi_{n}^{2}dx &=p \varepsilon_{n}^{N} \int_{\mathbb{R}^{N}} U^{p-1} \tilde{\varphi}_{n}^{2}dx \\
&=p \varepsilon_{n}^{N}\left(\int_{B_{R}(0)} U^{p-1} \tilde{\varphi}_{n}^{2}dx+\int_{\mathbb{R}^{N} \backslash B_{R}(0)} U^{p-1} \tilde{\varphi}_{n}^{2}dx\right) \\
&=p \varepsilon_{n}^{N}\left(o(1)+o_{R}(1)\right)
\end{aligned}
\end{equation*}
where $o(1) \rightarrow 0$ as $n \rightarrow \infty$ since $\tilde{\varphi}_{n} \rightarrow 0$ in $L_{\text {loc }}^{p+1}\left(\mathbb{R}^{N}\right)$, and $o_{R}(1) \rightarrow 0$ as $R \rightarrow \infty$ since $\tilde{\varphi}_{n} \in$ $H_{\varepsilon}$ is uniformly bounded. Take $R$ sufficiently large. We get
\begin{equation*}
p \int_{\mathbb{R}^{N}} U_{\varepsilon_{n}, y_{n}}^{p-1} \varphi_{n}^{2}dx \leq \frac{1}{2} \varepsilon_{n}^{N}
\end{equation*}
for $n$ sufficiently large. However, this implies that
\begin{equation*}
\begin{aligned}
\frac{1}{n} \varepsilon_{n}^{N}=\frac{1}{n}\left\|\varphi_{n}\right\|_{\varepsilon_{n}}^{2} \geq &\left\langle\mathcal{L}_{\varepsilon_{n}} \varphi_{n}, \varphi_{n}\right\rangle \\
=&\left\|\varphi_{n}\right\|_{\varepsilon_{n}}^{2}+b \varepsilon_{n} \int_{\mathbb{R}^{N}}\left|(-\Delta)^{\frac{s}{2}} U_{\varepsilon_{n}, y_{n}}\right|^{2} \int_{\mathbb{R}^{N}}\left|(-\Delta)^{\frac{s}{2}} \varphi_{n}\right|^{2}dx \\
&+2 b \varepsilon_{n}\left(\int_{\mathbb{R}^{N}} (-\Delta)^{\frac{s}{2}} U_{\varepsilon_{n}, y_{n}} \cdot (-\Delta)^{\frac{s}{2}} \varphi_{n}dx\right)^{2}-p \int_{\mathbb{R}^{N}} U_{\varepsilon_{n}, y_{n}}^{p-1} \varphi_{n}^{2}dx \\
\geq & \frac{1}{2} \varepsilon_{n}^{N}.
\end{aligned}
\end{equation*}
We reach a contradiction. The proof is complete.
\end{proof}
\par
Lemma \ref{Lem5.1} implies that by restricting on $E_{\varepsilon, y}$, the quadratic form $\mathcal{L}_{\varepsilon}: E_{\varepsilon, y} \rightarrow E_{\varepsilon, y}$ has a bounded inverse, with $\left\|\mathcal{L}_{\varepsilon}^{-1}\right\| \leq \rho^{-1}$ uniformly with respect to $y \in B_{\delta}(x_0)$. This further implies the following reduction map.
\begin{Lem}\label{Lem5.2}
There exist $\varepsilon_{0}>0, \delta_{0}>0$ sufficiently small such that for all $\varepsilon \in\left(0, \varepsilon_{0}\right), \delta \in$ $\left(0, \delta_{0}\right)$, there exists a $C^{1}$ map $\varphi_{\varepsilon}: B_{\delta}(x_0) \rightarrow H_{\varepsilon}$ with $y \mapsto \varphi_{\varepsilon, y} \in E_{\varepsilon, y}$ satisfying
\begin{equation*}
\left\langle\frac{\partial J_{\varepsilon}\left(y, \varphi_{\varepsilon, y}\right)}{\partial \varphi}, \psi\right\rangle_{\varepsilon}=0, \quad \forall \psi \in E_{\varepsilon, y}.
\end{equation*}
Moreover, there exists a constant $C>0$ independent of $\varepsilon$ small enough and $\kappa\in(0,\frac{\alpha}{2})$ such that
\begin{equation*}
\|\varphi_{\varepsilon, y}\|_{\varepsilon} \leq C \varepsilon^{\frac{N}{2}+\alpha-\kappa}+C \varepsilon^{\frac{N}{2}} \left(V\left(y\right)-V\left(x_{0}\right)\right)^{1-\kappa}.
\end{equation*}
\end{Lem}
\begin{proof}This existence of the mapping $y \mapsto \varphi_{\varepsilon, y}$ follows from the contraction mapping theorem. We construct a contraction map as follows.
Let $\varepsilon_{1}$ and $\delta_{1}$ be defined as in Lemma \ref{Lem5.1}. Let $\varepsilon_{0} \leq \varepsilon_{1}$ and $\delta_{0} \leq \delta_{1}$. We will choose $\varepsilon_{0}$ and $\delta_{0}$ later. Fix $y \in B_{\delta}(x_0)$ for $\delta<\delta_{0}$. Recall that
\begin{equation*}
J_{\varepsilon}(y, \varphi)=I_{\varepsilon}\left(U_{\varepsilon, y}\right)+l_{\varepsilon}(\varphi)+\frac{1}{2}\left\langle\mathcal{L}_{\varepsilon} \varphi, \varphi\right\rangle+R_{\varepsilon}(\varphi).
\end{equation*}
So we have
\begin{equation*}
\frac{\partial J_{\varepsilon}(\varphi)}{\partial \varphi}=l_{\varepsilon}+\mathcal{L}_{\varepsilon} \varphi+R_{\varepsilon}^{\prime}(\varphi).
\end{equation*}
Since $E_{\varepsilon, y}$ is a closed subspace of $H_{\varepsilon}$, Lemma \ref{Lem4.4} and \ref{Lem4.5} implies that $l_{\varepsilon}$ and $R_{\varepsilon}^{\prime}(\varphi)$ are bounded linear operators when restricted on $E_{\varepsilon, y} .$ So we can identify $l_{\varepsilon}$ and $R_{\varepsilon}^{\prime}(\varphi)$ with their representatives in $E_{\varepsilon, y}$. Then, to prove Lemma \ref{Lem5.2}, it is equivalent to find $\varphi \in E_{\varepsilon, y}$ that satisfies
\begin{equation}\label{eq5.3}
\varphi=\mathcal{A}_{\varepsilon}(\varphi) \equiv-\mathcal{L}_{\varepsilon}^{-1}\left(l_{\varepsilon}+R_{\varepsilon}^{\prime}(\varphi)\right).
\end{equation}
To solve \eqref{eq5.3}, we set
\begin{equation*}
\begin{aligned}
S_{\varepsilon}:=&\left\{\varphi \in E_{\varepsilon,y}:\|\varphi\|_{\varepsilon} \leq \varepsilon^{\frac{N}{2}+\alpha-\kappa}+\varepsilon^{\frac{N}{2}}\left|V\left(y\right)-V\left(x_{0}\right)\right|^{1-\kappa}\right\}
\end{aligned}
\end{equation*}
for $\kappa\in(0,\frac{\alpha}{2})$ sufficiently small.
We shall verify that $\mathcal{A}_{\varepsilon}$ is a contraction mapping from $S_{\varepsilon}$ to itself. For $\varphi \in S_{\varepsilon}$, by Lemma \ref{Lem4.4} and \ref{Lem4.5}, we obtain
\begin{equation*}
\begin{aligned}
\|\mathcal{A}_{\varepsilon}(\varphi) \| \leq & C\left(\left\|l_{\varepsilon}\right\|+\left\|R_{\varepsilon}^{\prime}(\omega)\right\|\right) \\
\leq & C\left\|l_{\varepsilon}\right\|+C \varepsilon^{-\frac{N}{2}}\|\varphi\|^{2} \\
\leq & C\left(\varepsilon^{\frac{N}{2}+\alpha}+\varepsilon^{\frac{N}{2}} \left|V\left(y\right)-V\left(x_{0}\right)\right|+\varepsilon^{-\frac{N}{2}}\left(\varepsilon^{N+2\alpha-2 \kappa}+\varepsilon^{N} \left|V\left(y\right)-V\left(x_{0}\right)\right|^{2(1-\kappa)}\right)\right) \\
\leq & C\left(\varepsilon^{\frac{N}{2}+\alpha-\kappa}+\varepsilon^{\frac{N}{2}} \left|V\left(y\right)-V\left(x_{0}\right)\right|^{1-\kappa}\right) .
\end{aligned}
\end{equation*}
Then, we get $\mathcal{A}_{\varepsilon}$ maps $S_{\varepsilon}$ to $S_{\varepsilon}$.
On the other hand, for any $\varphi_{1}, \varphi_{2} \in S_{\varepsilon}$
\begin{equation*}
\begin{aligned}
\left\|\mathcal{A}_{\varepsilon}\left(\varphi_{1}\right)-\mathcal{A}_{\varepsilon}\left(\varphi_{2}\right)\right\| &=\left\|L_{\varepsilon}^{-1} R_{\varepsilon}^{\prime}\left(\varphi_{1}\right)-L_{\varepsilon}^{-1} R_{\varepsilon}^{\prime}\left(\varphi_{2}\right)\right\| \\
& \leq C\left\|R_{\varepsilon}^{\prime}\left(\varphi_{1}\right)-R_{\varepsilon}^{\prime}\left(\varphi_{2}\right)\right\| \\
& \leq C\left\|R_{\varepsilon}^{\prime \prime}\left(\theta \varphi_{1}+(1-\theta) \varphi_{2}\right)\right\|\left\|\varphi_{1}-\varphi_{2}\right\|_{\varepsilon} \\
& \leq \frac{1}{2}\left\|\varphi_{1}-\varphi_{2}\right\|_{\varepsilon}.
\end{aligned}
\end{equation*}
So, $\mathcal{A}_{\varepsilon}$ is a contraction map from $S_{\varepsilon}$ to $S_{\varepsilon}$. Thus, there exists a contraction map $y \mapsto$ $\varphi_{\varepsilon, y}$ such that \eqref{eq5.3} holds.
\par
At last, we claim that the map $y \mapsto \varphi_{\varepsilon, y}$ belongs to $C^{1}$. Indeed, by similar arguments as that of Cao, Noussair and Yan \cite{MR1686700}, we can deduce a unique $C^{1}$-map $\tilde{\varphi}_{\varepsilon, y}: B_{\delta}(x_0) \rightarrow E_{\varepsilon, y}$ which satisfies \eqref{eq5.3}. Therefore, by the uniqueness $\varphi_{\varepsilon, y}=\tilde{\varphi}_{\varepsilon, y}$, and hence the claim follows.
\end{proof}

\subsection{Proof of Theorem \ref{Thm1.3}}
Let $\varepsilon_{0}$ and $\delta_{0}$ be defined as in Lemma \ref{Lem5.2} and let $\varepsilon<\varepsilon_{0}$. Fix $0<$ $\delta<\delta_{0}$. Let $y \mapsto \varphi_{\varepsilon, y}$ for $y \in B_{\delta}(x_0)$ be the map obtained in Lemma \ref{Lem5.2}. As aforementioned in {\bf Step 2} in Section 3, it is equivalent to find a critical point for the function $j_{\varepsilon}$ defined as in \eqref{eq4.1} by Lemma \ref{Lem4.2}. By the Taylor expansion, we have
\begin{equation*}
j_{\varepsilon}(y)=J\left(y, \varphi_{\varepsilon, y}\right)=I_{\varepsilon}\left(U_{\varepsilon, y}\right)+l_{\varepsilon}\left(\varphi_{\varepsilon, y}\right)+\frac{1}{2}\left\langle\mathcal{L}_{\varepsilon} \varphi_{\varepsilon, y}, \varphi_{\varepsilon, y}\right\rangle+R_{\varepsilon}\left(\varphi_{\varepsilon, y}\right).
\end{equation*}
We analyze the asymptotic behavior of $j_{\varepsilon}$ with respect to $\varepsilon$ first.
\par
By Lemma \ref{Lem4.4}, Lemma \ref{Lem4.5}, Lemma \ref{Lem4.7} and Lemma \ref{Lem5.2}, we have
\begin{equation}\label{eq5.4}
\begin{aligned}
j_{\varepsilon}(y)&=I_{\varepsilon}\left(U_{\varepsilon, y}\right)+O\left(\left\|l_{\varepsilon}\right\|\left\|\varphi_{\varepsilon}\right\|+\left\|\varphi_{\varepsilon}\right\|^{2}\right)\\
&=A \varepsilon^{N}+B \varepsilon^{N}\left( V\left(y\right)-V\left(x_{0}\right)\right)+ O(\varepsilon^{N})\left( \varepsilon^{\alpha-\kappa}+ \left(V\left(y\right)-V\left(x_{0}\right)\right)^{1-\kappa}\right)^2+O(\varepsilon^{N+\alpha})\\
\end{aligned}
\end{equation}
\par
Now consider the minimizing problem
\begin{equation*}
j_{\varepsilon}\left(y_{\varepsilon}\right) \equiv \inf _{y \in B_{\delta}(x_0)} j_{\varepsilon}(y).
\end{equation*}
Assume that $j_{\varepsilon}$ is achieved by some $y_{\varepsilon}$ in $B_{\delta}(x_0) .$ We will prove that $y_{\varepsilon}$ is an interior point of $B_{\delta}(x_0)$.
\par
To prove the claim, we apply a comparison argument. Let $e \in \mathbb{R}^{N}$ with $|e|=1$ and $\eta>1$. We will choose $\eta$ later. Let $z_{\epsilon}=\epsilon^{\eta} e \in B_{\delta}(0)$ for a sufficiently large $\eta>1$. By the above asymptotics formula, we have
\begin{equation*}
\begin{aligned}
j_{\epsilon}\left(z_{\epsilon}\right)=& A \epsilon^{N}+B \epsilon^{N}\left(V\left(z_{\epsilon}\right)-V(0)\right)+O\left(\epsilon^{N+\alpha}\right) \\
&+O\left(\epsilon^{N}\right)\left(\epsilon^{\alpha-\kappa}+\left(V\left(z_{\epsilon}\right)-V(0)\right)^{1-\kappa}\right)^{2}.
\end{aligned}
\end{equation*}
Applying the H\"{o}lder continuity of $V$, we derive that
\begin{equation*}
\begin{aligned}
j_{\epsilon}\left(z_{\epsilon}\right)=& A \epsilon^{N}+O\left(\epsilon^{N+\alpha \eta}\right)+O\left(\epsilon^{N+\alpha}\right) \\
&+O\left(\epsilon^{N}\left(\epsilon^{2(\alpha-\tau)}+\epsilon^{2 \eta \alpha(1-\kappa)}\right)\right) \\
=& A \epsilon^{N}+O\left(\epsilon^{N+\alpha}\right).
\end{aligned}
\end{equation*}
where $\eta>1$ is chosen to be sufficiently large accordingly. Note that we also used the fact that $\kappa\ll \alpha / 2$. Thus, by using $j\left(y_{\epsilon}\right) \leq j\left(z_{\epsilon}\right)$ we deduce
\begin{equation*}
B \epsilon^{N}\left(V\left(y_{\epsilon}\right)-V(0)\right)+O\left(\epsilon^{N}\right)\left(\epsilon^{\alpha-\kappa}+\left(V\left(y_{\epsilon}\right)-V(0)\right)^{1-\kappa}\right)^{2} \leq O\left(\epsilon^{N+\alpha}\right)
\end{equation*}
That is,
\begin{equation}\label{eq5.5}
B\left(V\left(y_{\epsilon}\right)-V(0)\right)+O(1)\left(\epsilon^{\alpha-\kappa}+\left(V\left(y_{\epsilon}\right)-V(0)\right)^{1-\kappa}\right)^{2} \leq O\left(\epsilon^{\alpha}\right)
\end{equation}
If $y_{\epsilon} \in \partial B_{\delta}(0)$, then by the assumption $(V_2)$, we have
\begin{equation*}
V\left(y_{\epsilon}\right)-V(0) \geq c_{0}>0
\end{equation*}
for some constant $0<c_{0} \ll 1$ since $V$ is continuous at $x=0$ and $\delta$ is sufficiently small. Thus, by noting that $B>0$ from Lemma \ref{Lem4.7} and sending $\epsilon \rightarrow 0$, we infer from \eqref{eq5.5} that
\begin{equation*}
c_{0} \leq 0.
\end{equation*}
We reach a contradiction. This proves the claim. Thus $y_{\epsilon}$ is a critical point of $j_{\epsilon}$ in $B_{\delta}(x_0)$. Then Theorem \ref{Thm1.3} now follows from the claim and Lemma \ref{Lem4.2}.

\section*{Acknowledgments}
The research of Vicen\c{t}iu D. R\u{a}dulescu was supported by a grant of
the Romanian Ministry of Research, Innovation and Digitization, CNCS/CCCDI-UEFISCDI,
project number PCE 137/2021, within PNCDI III.
The research of Zhipeng Yang was supported by the RTG 2491 "Fourier Analysis and Spectral Theory".

\bibliographystyle{plain}
\bibliography{yang}

\end{document}